\documentclass[10pt]{amsart}
\usepackage{amsmath,amsthm,amssymb,amsfonts, eucal, amscd, mathbbol, mathrsfs, mathabx}
\usepackage[shortlabels]{enumitem}
\usepackage{cite}
\usepackage{setspace}
\usepackage[all,cmtip]{xy}{
\usepackage{graphicx}
\usepackage{subfig}
\usepackage{fancyhdr}
\usepackage{latexsym}
\usepackage{fncylab}
\usepackage{epic}
\usepackage{ifthen}
\usepackage[bookmarks,bookmarksnumbered, plainpages=false, pdfpagelabels]{hyperref}
\usepackage{xcolor}
\hypersetup{
 colorlinks = true, %Colours links instead of ugly boxes
 urlcolor = green, %Colour for external hyperlinks
 linkcolor = blue, %Colour of internal links
 citecolor = red %Colour of citations
}

\usepackage{rotating}
\usepackage{scalefnt}
\usepackage{enumitem}
\setlist{nolistsep}

\raggedbottom

\parskip = 0.2in
\parindent = 0.0in
\topmargin = 0.0in

\oddsidemargin = 0.0in
\evensidemargin = 0.0in
\textwidth = 6.0in

\newtheorem{thm}{Theorem}[section]
\newtheorem*{thm*}{Main Theorem}
\newtheorem*{thm**}{Theorem}

\newtheorem{cor}[thm]{Corollary}
\newtheorem{lem}[thm]{Lemma}

\theoremstyle{definition}
\newtheorem{defn}[thm]{Definition}
\theoremstyle{definition}

\theoremstyle{definition}

\theoremstyle{definition}

\theoremstyle{definition}
\newtheorem{examples}[thm]{Examples}
\theoremstyle{definition}

\theoremstyle{definition}

\numberwithin{thm}{subsection}

\newcommand{\R}{\ensuremath{\mathbb{R}}}
\newcommand{\N}{\ensuremath{\mathbb{N}}} 

\newcommand{\F}{\ensuremath{\mathbb{F}}}

\def\p{\partial}
\def\i{\infty}
 
\def\supp{\it{supp}}

\def\frak{\mathfrak}
\def\diam{\emph{diam}} 

\def\fr{\p}
 
\def\image{\emph{im}}

\def\cal{\mathcal}
\def\a{\alpha}
 
\def\g{\gamma} 

\def\d{\delta}
\def\D{\Delta} 
\def\e{\epsilon}

\def\H{\mathcal{H}}

\def\F{\mathcal{F}}

\def\L{\mathcal{L}}

\def\cN{\cal{N}}
\def\lip{\mathrm{Lip}}

\def\Span{\mathcal{S}}

\def\Gr{\mathrm{Gr}}
\tracinggroups=1

\makeatletter
\newcommand\footnoteref[1]{\protected@xdef\@thefnmark{\ref{#1}}\@footnotemark}
\makeatother

\def\XXint#1#2#3{{\setbox0=\hbox{$#1{#2#3}{\int}$}
\vcenter{\hbox{$#2#3$}}\kern-.5\wd0}}

\renewcommand{\theenumi}{(\alph{enumi})}
\makeatletter
\renewcommand{\p@enumii}{\theenumi}
\makeatother

\begin{document} 
	
\author[J. Harrison \& H. Pugh]{J. Harrison \\Department of Mathematics \\University of California, Berkeley \\ H. Pugh\\ Mathematics Department \\ Stony Brook University} 
\title{General Methods of Elliptic Minimization}
\begin{abstract}
	We provide general methods in the calculus of variations for the anisotropic Plateau problem in arbitrary dimension and codimension. Given a collection of competing ``surfaces,'' which span a given ``bounding set'' in an ambient metric space, we produce one minimizing an elliptic area functional. The collection of competing surfaces is assumed to satisfy a set of geometrically-defined axioms. These axioms hold for collections defined using any combination of homological, cohomological or linking number spanning conditions. A variety of minimization problems can be solved, including sliding boundaries.
\end{abstract}
\maketitle
	
\section{Introduction}
	\label{sec:introduction}
	Plateau's problem asks if there exists a surface of least area among those with a given boundary. It was named after the French physicist Joseph Plateau, who in the 19\textsuperscript{th} century experimented with soap films and formulated laws that describe their structure. There is no single theorem or conjecture called Plateau's problem; it is rather a general framework which has many precise formulations. Douglas and Rad\'{o} \cite{douglas,rado} independently solved the first such formulation of Plateau's problem by finding an area minimizer among immersed parametrized disks with a prescribed boundary in \( \R^n \). Three seminal papers appearing in 1960 \cite{reifenberg,federerfleming,degiorgi} employed different definitions of ``surface'' and ``boundary'' and solved distinct versions of the problem. The techniques developed in these papers gave birth to the modern field of geometric measure theory.
	
	We will briefly mention the problems solved in \cite{reifenberg} and \cite{federerfleming}, leaving proper definitions and details to the original sources. An \( m \)-rectifiable set equipped with a pointwise orientation and integer multiplicity can be integrated against differential forms. Such objects are called ``rectifiable currents'' and possess mass (\( m \)-dimensional Hausdorff measure weighted by the multiplicity) and a boundary operator (the dual to exterior derivative.) If the boundary of a rectifiable current is also rectifiable, it is called an ``integral current.'' Federer and Fleming \cite{federerfleming} used these integral currents to define their competing surfaces and used mass to define ``area.'' On the other hand, Reifenberg \cite{reifenberg} used compact sets for surfaces and Hausdorff measure to define area. There is no boundary operator defined for sets, so instead he turned to \v{C}ech homology to define a collection of competing ``surfaces'' with a given boundary. Roughly speaking, given a boundary set \( A \) and a set \( L \) of \v{C}ech cycles in \( A \), a set \( X\supset A \) is a competing surface if each cycle in \( L \) bounds in \( X \).
	
	There are advantages and disadvantages to using either sets or currents. Each approach has its own applications and is suitable for different problems. Sets tend to be more difficult to work with than currents because of the lack of a boundary operator and the fact that unlike mass, Hausdorff measure is not lower-semicontinuous in any useful topology. A substantive difference between the two is that integral currents possess an orientation and sets do not. In practice, two currents with opposite orientations cancel when brought together, while sets do not. Sets are better models for physical soap films, since if two soap films touch, they merge rather than cancel.

	Plateau's problem requires minimization of an area functional \( X \mapsto \mathit{A}(X) \) where \( X \) is a competing surface and \( \mathit{A} \) refers to either mass or \( m \)-dimensional Hausdorff measure. The problem can be generalized to a heterogeneous problem by allowing the ambient density of \( \mathit{A} \) to vary pointwise by a function \( f \). In this case, one would minimize the functional \( X \mapsto \int_X f(p)\mathit{dA} \). The heterogeneous problem itself is a special case of an anisotropic minimization problem in which the density function \( f \) can depend non-trivially on \( m \)-dimensional tangent directions. In this case, the functional would be \( X\mapsto \int_X f(p,T_pX)\mathit{dA} \). 
	
	For example, consider the cost of building roads between several towns. If the land is flat and homogeneous but with a varying cost of acquisition, then the cost minimization problem is a heterogenous but isotropic minimization problem. However, if the land is hilly with variable topography, then cost becomes an anisotropic problem.
 	 
	Almgren \cite{almgrenannals} worked on the anisotropic minimization problem and defined a necessary ellipticity condition on the area functional. Roughly speaking, an anisotropic area functional is elliptic if an \( m \)-disk centered at any given point can be made small enough so that it very nearly minimizes the area functional among surfaces with the same \( (m-1) \)-sphere boundary. This ellipticity condition as defined in \cite{almgrenannals} is analogous to Morrey's quasiconvexity used in parametric variational problems \cite{morreyconvex}. Federer \cite{federer} used a parametric variant of Almgren's elliptic integrands to obtain an anisotropic version of \cite{federerfleming} for integral currents.

	In this paper we establish the existence of an \( m \)-dimensional surface in an ambient metric space which minimizes an elliptic area functional for collections of sets satisfying axiomatic spanning conditions, including the collections considered in \cite{almgrenannals}. This solves a problem of geometric measure theory from the 1960's (e.g., see \cite{almgrenannals},\cite{almgren},) namely to provide an elliptic version of the ``size\footnote{``Size'' in this context refers to Hausdorff measure. Physically realistic models of soap films use size instead of mass to define area.} minimization problem'' as in Reifenberg \cite{reifenberg}. Roughly speaking, given a bounding set \( A \) and a collection of \( m \)-rectifiable sets \( X \) which ``span'' \( A \) with respect to a geometrically-defined set of axioms \S\ref{axioms}, we prove there exists an element in the collection with minimum \( m \)-dimensional Hausdorff measure, weighted by an anisotropic density function. In \S\ref{sub:axiomatic_spanning_conditions} we describe a variety of topologically-defined collections which satisfy the axioms. Our methods build upon the isotropic results in \cite{lipschitzarxiv,lipschitz}.

	\subsection{Recent history and current developments} % (fold)
		\label{sec:recent_hisoty_and_current_developments}
		 
		In \cite{hpplateau} we used linking numbers to specify spanning conditions: If \( M \) is an oriented \( (n-2) \)-dimensional connected submanifold of \( \R^n \), we say a set \( X\subset \R^n \) \emph{\textbf{spans}} \( M \) if every circle embedded in the complement of \( M \) which has linking number one with \( M \) has non-trivial intersection with \( X \). This definition can be extended to arbitrary codimension by replacing linking circles with spheres and to the case that \( M \) is not connected by specifying linking numbers with each component. We proved the following result, relying on \cite{almgren} for regularity: 
		
	 	\textbf{Theorem} \cite{hpplateau}
		\mbox{}\\
			 Let \( M \) be an oriented, compact \( (n-2) \)-dimensional submanifold of \( \R^n \) and \( \Span \) the collection of compact sets spanning \( M \). There exists an \( X_0 \) in \( \Span \) with smallest size. Any such \( X_0 \) contains a ``core'' \( X_0^*\in \Span \) with the following properties: It is a subset of the convex hull of \( M \) and is a.e. (in the sense of \( (n-1) \)-dimensional Hausdorff measure) a real analytic \( (n-1) \)-dimensional minimal submanifold. 
	 
		De Lellis, Ghiraldin and Maggi \cite{delellisandmaggi} built upon our linking number spanning condition and extracted more general axiomatic spanning conditions, a possibility first envisioned by David. Their beautiful work gave a new proof to the main result of \cite{hpplateau} and, simultaneously, a new approach to the ``sliding boundary'' problem also posed by David \cite{david2014} (see \S\ref{examples:span} for further discussion.) De Philippis, de Rosa and Ghiraldin in \cite{ghiraldin} extended their paper to higher codimension, replacing links by simple closed curves with links by spheres.

		In \cite{lipschitzarxiv,lipschitz}, we extended \cite{hpplateau} to higher codimension using a spanning condition defined using cohomology. We also minimized Hausdorff measure weighted by an isotropic H\"older density function. By Alexander duality, taking geometric representatives for homology classes, this cohomological spanning condition is equivalent to the above linking condition, but in which the linking spheres are replaced with surfaces with possibly higher genus and conical singularities.

		The isotropic density function of \cite{lipschitzarxiv,lipschitz} was replaced by an anisotropic density in the current paper. At essentially the same time as this paper was announced, de Lellis, de Rosa and Ghiraldin posted \cite{DeLDeRGhi16} for codimension one. Our two approaches use different axiomatic spanning conditions. Our axioms \S\ref{axioms} are, roughly speaking, that our collections of sets are closed under the action of diffeomorphisms keeping the bounding set \( A \) fixed and Hausdorff limits. We note that all collections using homological, cohomological, or linking spanning conditions satisfy these conditions. The axioms of \cite{DeLDeRGhi16}, similar to those in its predecessor \cite{delellisandmaggi}, use so-called ``good classes'' and ``deformation classes.'' Roughly speaking, ``good classes'' use cup competitors arising from Caccioppoli theory and ``deformation classes'' have to do with behavior under Lipschitz deformations. (We refer to \cite{DeLDeRGhi16} for the full definitions.) Neither cup nor deformed competitors are assumed to be included in the original collection, but ``are approximable in energy'' by elements of the collection.
		
	\subsection{Advances in this paper}\label{advances}
		
		\subsubsection{\textbf{Ambient spaces}}
			\label{ambient}
			\begin{figure}[h]
				\centering
				\includegraphics[width=.6\textwidth]{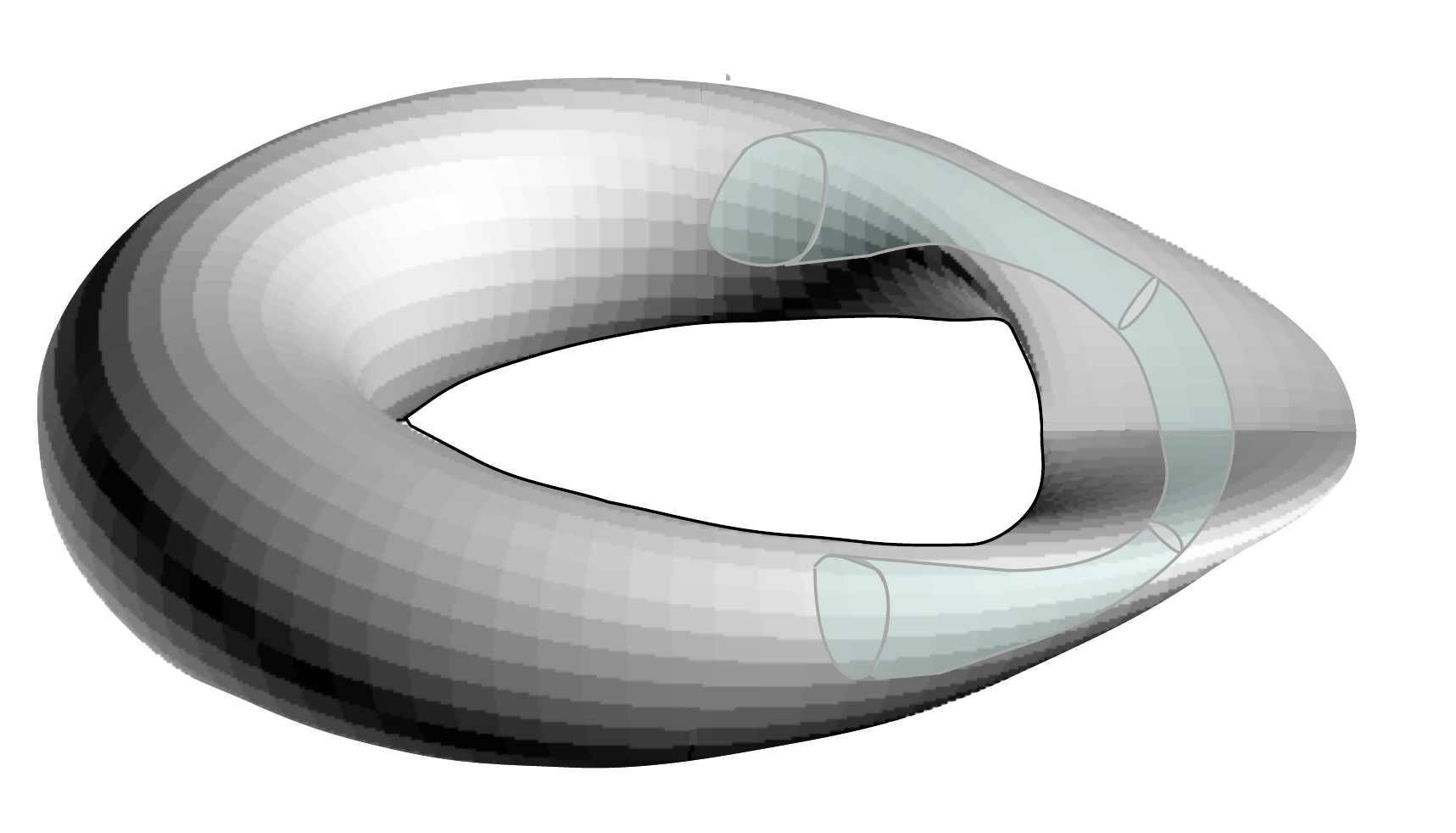}
				\caption{This pinched solid torus is an ambient space for anisotropic minimization problems. The bounding set \( A \) is the union of two (dashed) circles in the interior of the torus and the interior pinched (dashed) cylinder depicts a competitor spanning \( A \).}
				\label{fig:pinched2}
			\end{figure}
			
			In this paper we permit the ambient space in which the minimization occurs to be a certain type of metric space which can be isometrically embedded in \( \R^n \) as a Lipschitz retraction of some neighborhood of itself. See Definition \ref{def:LNR}. Examples include Riemannian manifolds with boundary and/or conical singularities. (See Figure \ref{sub:axiomatic_spanning_conditions}.)
		
			Whenever one wishes to extend a particular result in geometric measure theory from Euclidean space to an ambient Riemannian manifold, one is faced with the choice of either working in charts, or embedding the manifold in Euclidean space and proving that the various constructions used can be deformed back onto the manifold. Indeed, this second approach is usually much simpler, and yields further generalization to spaces more general than manifolds, namely Lipschitz neighborhood retracts. We have, for the most part, chosen this second approach. However, the full category of Lipschitz neighborhood retracts seems slightly out of reach. We make use of one construction in particular, namely Lemma \ref{cor:deformation}, in which it is vital to assume slightly more about the ambient space, namely that the Lipschitz retraction can be ``localized'' (Definition \ref{def:LNR}.) Nevertheless, this slightly restricted class of Lipschitz neighborhood retracts contains all the interesting examples we can think of, including Riemannian manifolds with boundary and certain singularities.

		\subsubsection{\textbf{Bounding sets}}
			\label{bounding}
			In our approach to the size minimization problem, we begin with a fixed ``bounding set'' \( A \). This is a compact set that all competitors \( X \) are assumed to contain\footnote{Or alternatively, all competitors \( X \) are assumed to be relatively closed subsets of the complement of \( A \). This is a stylistic choice; there are pros and cons to both approaches, but they are mathematically equivalent.}. The role that \( A \) plays in Plateau problems mimics that of a boundary condition in PDE's, but we call \( A \) a ``bounding set'' since it might look nothing like the boundary of an \( m \)-dimensional set in some simple examples (see Figure \ref{fig:sliding2}.) We permit \( A \) to be \emph{any} compact set, including the empty set and sets with dimension \( n \) (we minimize the elliptic functional over the sets \( X\setminus A \).) 
		\subsubsection{\textbf{Geometrically defined axiomatic spanning conditions}}
			\label{axioms} 
			 The problem of finding axiomatic conditions on collections of sets sufficient to solve minimization problems was posed in \cite{shallwe}. The axioms presented in \S\ref{sub:axiomatic_spanning_conditions} assume that our collections of sets are closed under Lipschitz deformations and Hausdorff limits. These conditions are all met by the algebraic spanning conditions in \ref{spanning}. If the ambient space is a Riemannian manifold, then the Lipschitz deformations can be replaced by diffeomorphisms isotopic to the identity (see Definition \ref{def:metricaxioms}.)
		\subsubsection{\textbf{Algebraic spanning conditions}}
			\label{spanning}
			Currents possess a boundary operator, and for minimization problems in which the ``surfaces'' are currents, this boundary operator can be used to specify a spanning condition. That is, a current \( S \) is said to span a current \( T \) if the boundary of \( S \) is \( T \). However, there is no boundary operator for sets and it takes more work to specify spanning conditions for minimization problems involving sets. We are aware of two closely related types of algebraic spanning conditions which satisfy our axioms. The first is defined using \v{C}ech theory, either homological, cohomological or a combination of the two (see Definition \ref{def:algebraicspanning}.). The second uses linking numbers as defined in \S\ref{sec:recent_hisoty_and_current_developments}, and is homotopical in nature (see Definition \ref{linking}.) The key property needed for both types is continuity under either weak or Hausdorff limits. That is, if \( \{X_i\} \) is a minimizing sequence of surfaces satisfying a spanning condition, and \( X_i \) converges to \( X_0 \), then \( X_0 \) should also satisfy the spanning condition. The \v{C}ech theoretic spanning conditions satisfy this property due to the unique continuity property of \v{C}ech theory. The linking number spanning conditions satisfy the property due to the fact that null intersection of compact sets is an open condition. See Definition \ref{def:algebraicspanning} for more details and Figure \ref{fig:BlueCubeBall} for an illustrative example.

	\subsection{Methods}
	 	\label{sec:methods}
		
		Our methods are those of classical geometric measure theory, drawing tools from Besicovitch \cite{besicovitch, besicovitchI,besicovitchII}, Reifenberg \cite{reifenberg}, Federer and Fleming \cite{federerfleming}, \cite{fleming}, and our previous work \cite{hpplateau,lipschitzarxiv,lipschitz}. We do not use quasiminimal sets, varifolds or currents at any stage in our proof, nor do we reference any results which require their use. Our proof of rectifiability is not based on density and Preiss's theorem, as was our isotropic result \cite{lipschitzarxiv,lipschitz}, but rather Federer-Fleming and Besicovitch-Federer projections.
			
		\par{\textbf{Reifenberg regular sequences}} Reifenberg did not use quasiminimal sets and thus did not encounter their problems, e.g., \cite{almgrenannals}\footnote{The authors believe that the existence proof in \cite{almgrenannals} contains a flaw. If there is control over the Hausdorff measures of deformations of the elements of a minimizing sequence \( \{X_k\}_{k\in \N} \), and this control is uniform across all scales and all \( k \), then the job of analyzing the limit set becomes markedly easier. This condition is called ``\((M,0,\d)\)-minimizing'' in \cite{almgren} and ``uniformly quasiminimal'' by others. Such sequences have nice properties; in particular, their limit sets are \( m \)-rectifiable with good bounds on density ratios. However, despite efforts by experts in the field there is as of yet no known method to convert an arbitrary minimizing sequence \( \{X_k\} \) into a uniformly quasiminimal one. There is a non-trivial gap in \cite{almgrenannals} where Almgren assumed that a minimizing sequence for the elliptic integrand is uniformly quasiminimal. (See the last paragraph of 2.9(b2) which is needed for the main existence theorem. In this section he is working with a compact rectifiable subset \( S \) and assumes it is quasiminimal immediately before the conclusion of 2.9(b2) which is the isoperimetric inequality. The quasiminimal constant assumed here must be uniform across all scales and all \( S \) as can be seen in the way he applies the isoperimetric inequality. Interested readers could start with 3.4 where he introduces a minimizing sequence for the first time. His proof of lower density bounds uses 2.9(b2), but all he has is a minimizing sequence at this point and he cannot apply 2.9(b2). Indeed, it is not hard to come up with minimizing sequences that are not uniformly quasiminimal.) An indication that Almgren may have been aware of this problem is found in his last major paper on the anisotropic Plateau’s problem \cite{almgren}, Almgren assumed that his competitors were a priori quasiminimal, thus filling the gap, but in so doing gave up a more general existence theorem \cite{morgan}. His solution depends on the quasiminimal constant, and compactness fails if this constant is permitted to vary. In his review article \cite{questionsandanswers} he took appropriate credit for his important definition of elliptic integrands and his proof of regularity for minimizing solutions in \cite{almgrenannals} and \cite{almgren}, but he did not claim to have proved an existence theorem.}. The authors found a key definition buried in a proof of \cite{reifenberg} and overlooked until now. ``Reifenberg regular sequences'' are sequences \( \{X_k\}_{k\in \N} \) in which \( X_k \) has a uniform lower bound on density ratios, down to a scale that decreases as \( k \) increases (see Definition \ref{def:rregular}). Given a minimizing convergent sequence, it is not too hard to produce a Reifenberg regular subsequence. Limits of these sequences have nice properties (see \cite{lipschitz} \S4.3,) most importantly the possession of a uniform lower density bound. The main result of the current paper boils down to showing that two key constructions of Fleming (see Lemma \ref{lem:upperdensity} and \cite{fleming} 8.2,) and Almgren (see Theorem \ref{thm:rectifiable} and \cite{almgrenannals} 3.2(c)) can be applied to Reifenberg regular minimizing sequences.

	\subsection{Acknowledgements}
		\label{sub:remarks}
		
		We wish to thank the reviewer for his insights and suggestions for improving our paper and Francesco Maggi for his helpful comments on the introduction.

\section*{Notation}
	\label{sec:notation}
	Notation and terminology follow \cite{mattila} for the most part. If \( X\subset \R^n \),
	\begin{itemize}
		\item \( \bar{X} \) is the closure of \( X \);
		\item \( \mathring{X} \) is the interior of \( X \);
		\item \( X^c \) is the complement of \( X \);
		\item \( X^* \) is the core of \( X \);
		\item \( \cN(X,\e) \) is the open epsilon neighborhood of \( X \);
		\item \( B(X,\e) \) is the closed epsilon neighborhood of \( X \);
		\item \( d_H(\cdot,\cdot) \) is the Hausdorff distance;
		\item \( \H^m(X) \) is the \( m \)-dimensional Hausdorff measure of \( X \);
		\item \( X(p,r) = X \cap B(p,r) \);
		\item \( x(p,r) = X \cap \fr\,B(p,r) \);
		\item \( C_p(X) \) is the (inward) cone over \( X \) with basepoint \( p \);
		\item \( \a_m \) is the Lebesgue measure of the unit \( m \)-ball in \( \R^m \);
		\item \( \Gr(m,n) \) is the Grassmannian of un-oriented \( m \)-planes through the origin in \( \R^n \).
	\end{itemize}

\section{Definitions and Main Result}
	\label{sec:mainresults}
	\begin{defn}
		\label{def:LNR}
		A metric space \( C \) is a \emph{\textbf{Lipschitz neighborhood retract}} if there exists an isometric embedding \( C\hookrightarrow \R^n \) for some \( n>0 \), together with a neighborhood \( U\subset \R^n \) of \( C \) and a Lipschitz retraction \( \pi:U\to C \) (to simplify notation, we identify \( C \) with its image under the embedding.) We say a Lipschitz neighborhood retract is \emph{\textbf{localizable}} if there exists an embedding as above, such that for every \( p\in C \) there exist \( \kappa_p<\i \) and \( \xi_p>0 \) such that if \( 0<r<\xi_p \) then there exists a Lipschitz retraction \( \pi_{p,r}: C\cup B(p,r)\to C \) with \( \pi_{p,r}(B(p,r))=C(p,r) \) and with Lipschitz constant \( \le \kappa_p \). We call \( \xi_p \) the \emph{\textbf{retraction radius of \( C \) at \( p \)}}. If \( C \) is compact, the condition of being localizable implies that \( C \) is a priori a Lipschitz neighborhood retract. Localizability also implies local contractibility. If \( \kappa:=\sup_{p \in C} \{\kappa_p\}<\i \), then we say the localizable Lipschitz neighborhood retract is \emph{\textbf{uniform}}.
	\end{defn}
	
	For example, a Riemannian manifold is a uniform localizable Lipschitz neighborhood retract\footnote{More precisely, given an isometric embedding of Riemannian manifolds \( M\hookrightarrow \R^n \), we consider \( M \) with the pullback metric space structure induced by the embedding, which for our purposes is equivalent to the usual length metric space structure, since the corresponding Hausdorff measures will be identical.}.

	Let \( C \) be a metric space and suppose for a moment that we have a fixed isometric embedding \( C\subset \R^n \). For \( 1\leq m\leq n \), let \( T^mC \) denote the subbundle of the restriction to \( C \) of the unoriented Grassmannian bundle \( \R^n\times \Gr(m,n)\to \R^n \) consisting of pairs \( (p,E)\in C\times \Gr(m,n) \) such that \( E \) is the unique approximate tangent space at \( p \) for some \( \H^m \) measurable \( m \)-rectifiable subset \( X \) of \( C \) with \( \H^m(X)<\i \). Let \( T_p^m C \) denote the fiber of \( T^m C \) above \( p \).
		
	Let \( 0<a\leq b<\i \) and suppose \( f:T^mC\to [a,b] \) is measurable (for the Borel \( \sigma \)-algebra on \( T^mC \).) For an \( \H^m \) measurable \( m \)-rectifiable set \( X\subset C \), define \[ \F^m(X)= \int_X f(q,T_q X) d\H^m, \] where \( T_q X \) denotes the unique tangent \( m \)-plane to \( X \) at \( q \).

	Note that \( \F^m \) is defined independently of the the isometric embedding of \( C \) into Euclidean space. 
	\subsection{Ellipticity}
		\label{sub:ellipticity}
		\begin{defn}
			\label{def:elliptic}
			Let \( C \) be a uniform localizable Lipschitz neighborhood retract. We say \( \F^m \) is \emph{\textbf{elliptic}} if there exists an embedding of \( C \) into \( \R^n \) as a uniform localizable Lipschitz neighborhood retract, such that for every \( \H^m \) measurable \( m \)-rectifiable subset \( X \) of \( C \) with \( \H^m(X)<\i \), the following condition is satisfied for \( \H^m \) almost every \( p\in X \) such that \( X \) has a unique tangent \( m \)-plane \( E \) at \( p \): If \( \e>0 \), there exists \( s>0 \) such that if \( 0<r<s \), then
			\begin{align}
				\label{eq:elliptic}
				(f(p,E)-\e )H^m(E(p,r)) \le \F^m(Z\cap C) + b \H^m (Z\setminus C),
			\end{align}
			for every \( m \)-rectifiable closed set \( Z\subset B(p,r) \) such that
			\begin{enumerate}
				\item \( Z\cap \fr B(p,r) = E\cap \fr B(p,r) \); and
				\item\label{elliptic:3} There is no retraction from \( Z \) onto \( E\cap \fr B(p,r) \).
				\end{enumerate}	
			\end{defn}
	
			This definition captures Almgren's elliptic functionals (\cite{almgrenannals} 1.2) and generalizes them to a broader class of domains. In particular, \( C \) may be a region in \( \R^n \) with manifold boundary, or a manifold with singularities (see Figure \ref{fig:pinched2}.) See \S\ref{moreellipticty} for more on the ellipticity condition. If \( C \) is a Riemannian manifold, and the ellipticity condition holds for a particular embedding (of Riemannian manifolds) into \( \R^n \), then it will hold for all such embeddings. 
	
	\subsection{Spanning conditions}
		\label{sub:axiomatic_spanning_conditions}
		
		\begin{defn}
			Let \( C \) be a metric space, \( m\in \N \) and \( A\subset C \) be closed (possibly empty.) If \( X\subset C \), let \( X^* \) denote the subset of \( X \) consisting of points \( p\in X \) such that \( \H^m(X(p,r))>0 \) for all \( r>0 \). We say that \( X \) is \emph{\textbf{reduced}} if \( X^*=X \). If \( X\supset A \), let \( X^\dagger \) denote the set \( (X\setminus A)^*\cup A \). We say that \( X\supset A \) is a \emph\textbf{{surface}} if \( X \) is closed and \( X\setminus A \) is \( m \)-rectifiable, reduced, and \( \H^m(X\setminus A)<\i \). 
		\end{defn}
		
		\begin{defn}[\textbf{Axiomatic Spanning Conditions}]
			\label{def:metricaxioms}
			Let \( \Span(C,A) \) denote a collection of surfaces. We say \( \Span(C,A) \) is a \emph{\textbf{spanning collection}} if the following axioms hold:
			\begin{enumerate} 
				\item\label{lip} If \( g:C\to C \) is a Lipschitz map which fixes \( A \) and is homotopic to the identity relative to \( A \), and if \( X\in \Span(C,A) \), then \( g(X)^\dagger\in \Span(C,A) \).
				\item\label{lim} If \( \{X_k\}_{k\in \N}\subset \Span(C,A) \) and \( X_k \to X_0 \) in the Hausdorff distance, and if \( X_0\setminus A \) is \( m \)-rectifiable and satisfies \( \H^m(X\setminus A)<\i \), then \( X_0^\dagger \in \Span(C,A) \).
			\end{enumerate}
			We will also call \( \Span(C,A) \) a spanning collection in the case that \( C \) is a Riemannian manifold if in place of Axiom \eqref{lip}, the following weaker axiom holds:
			\begin{enumerate}
				\item[\eqref{lip}'] If \( g: C\to C \) is a diffeomorphism which fixes \( A \) and is isotopic to the identity relative to \( A \), and if \( X\in \Span(C,A) \), then \( g(X)\in \Span(C,A). \)
			\end{enumerate}
		\end{defn}
		
		Our main result is the following:
		
 		\begin{thm}
 			\label{thm:main}
			Suppose that \( C \) is a compact uniform localizable Lipschitz neighborhood retract and that \( A\subset C \) is closed (possibly empty.) Let \( m\in \N \). If \( \F^m \) is elliptic and if \( \Span(C,A) \) is a non-empty spanning collection, then \( \Span(C,A) \) contains an element which minimizes the functional \( X\mapsto \F^m(X\setminus A) \) among elements of \( \Span(C,A) \).
		\end{thm}
		 
		We next provide some examples of spanning collections.\\
	
		\begin{examples}
			\label{examples:span}
			\begin{itemize}\mbox{}\\
				\item \textbf{Algebraic spanning conditions}
			 
				\begin{defn}
					\label{def:algebraicspanning}
					Suppose \( (C,A) \) is a compact pair. An \emph{\textbf{algebraic spanning condition}} is a subset \( \L=(L_1,L_2,L_3,L_4) \) of \( H_{m-1}(A)\sqcup H^{m-1}(A)\sqcup H_m(C)\sqcup H^m(C) \). These are understood to be reduced \v{C}ech (co)homology groups, and the coefficients may vary between the four so long as the homology groups have compact coefficients. Let \( X\subset C \) be compact with \( X\supset A \) and let \( i: A\to X \) and \( j: X\to C \) denote the inclusions. We say \( X \) \emph{\textbf{spans \( \L \)}} if \( L_1\subset \ker(i_*) \), \( L_2\cap \image(i^*)=\emptyset \), \( L_3\subset \image(j_*) \), and \( L_4\cap \ker(j^*)=\emptyset \). Let \( \Span(C,A,\L) \) denote the set of all surfaces which span \( \L \).
				\end{defn}
				
				It follows from standard \v{C}ech theory that \( \Span(C,A,\L) \) is a spanning collection. See \cite{lipschitz} Lemmas 1.2.17 and 1.2.18, and Adams' appendix in \cite{reifenberg} for details.\\
			
				For example, suppose \( C = B(0,1) \setminus B(0,1/5) \subset \R^3\) and \( A \) is the cubical frame in Figure \ref{fig:BlueCubeBall}. Let \( L_1 \) be any element of \( H_1(A) \) and \( L_3 \) a generator of \( H_2(C) \). Let \( L_2 = L_4 = \emptyset \). Then the surface \( X \) in Figure \ref{fig:BlueCubeBall} is an element of \( \Span(C,A,\L) \).
			
				\begin{figure}[h]
					\centering	
					\includegraphics[width=.6\textwidth]{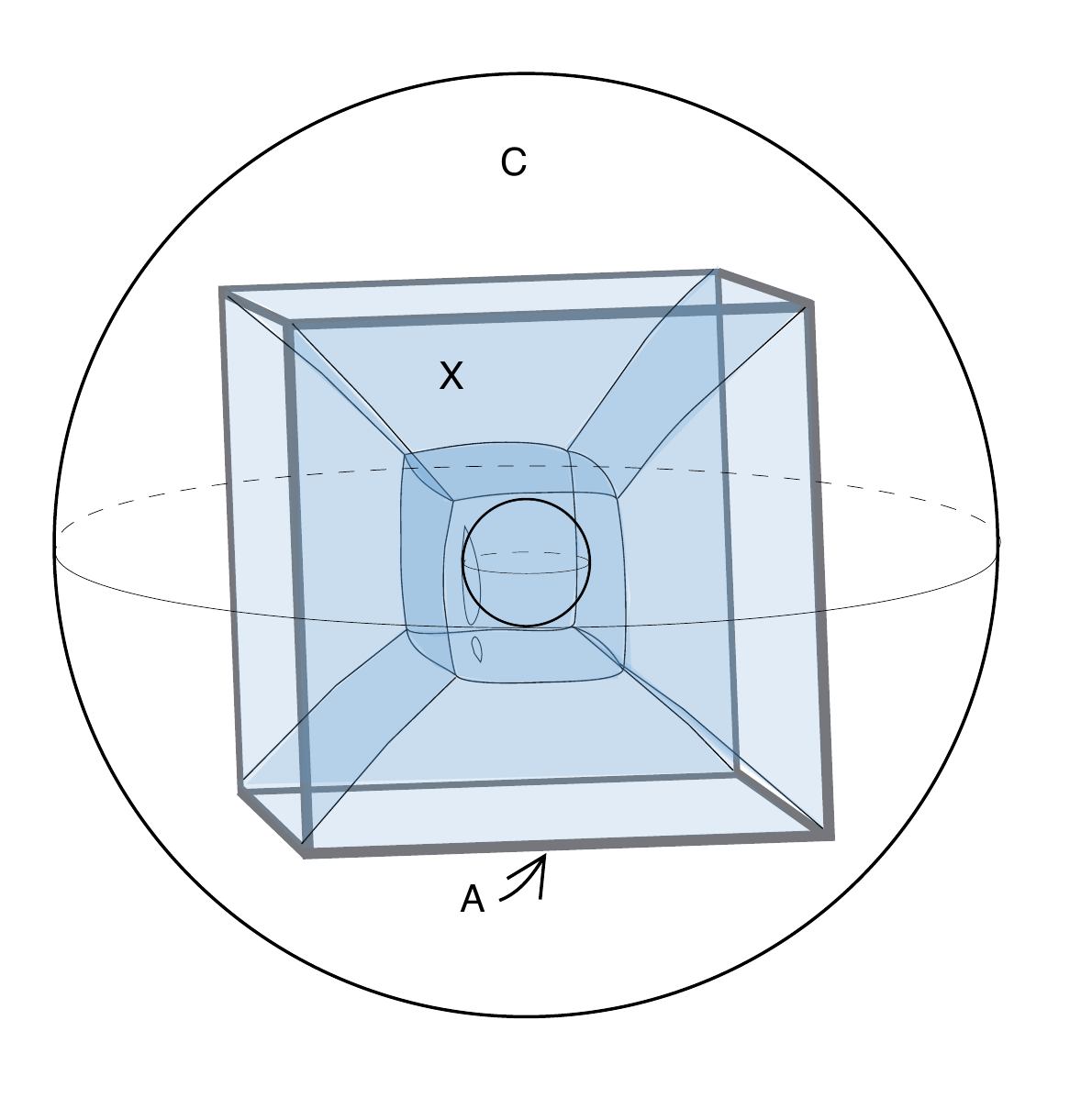}
					\caption{The surface \( X \) spans the cubical frame \( A \) within the ambient space \( C= B(0,1) \setminus B(0,1/5) \) under a number of different algebraic spanning conditions.}
					\label{fig:BlueCubeBall}
				\end{figure}
				\mbox{}\\
			 
				\item \textbf{Homotopical linking number spanning conditions}\\ 
				If \( C \) is a Riemannian manifold, a spanning collection may be defined using the homotopical ``linking number'' spanning condition defined by the authors in \cite{hpplateau}. The following definition generalizes this idea: 
			 
				\begin{defn}
					\label{linking}
					Suppose \( \cal{S} \) is a collection of compact, smoothly embedded manifolds \( M\subset C\setminus A \) which is invariant under diffeomorphisms which fix \( A \) and isotopic to the identity relative to \( A \). Let \( \Span(C,A) \) be the collection of all surfaces \( X \) which intersect non-trivially with every element of \( \cal{S} \). We say that elements of \( \Span(C,A) \) satisfy a \emph{\textbf{linking number spanning condition}}.
				\end{defn}
				It is straightforward to show that \( \Span(C,A) \) is a spanning collection.\\			 
			 
				\item \textbf{Sliding boundaries and minimizers}\\ 
				One may pick some initial surface \( X \), and define \( \Span(C,A) \) to be the smallest spanning collection which contains \( X \), and which is also closed under the action of Lipschitz functions \( g: (C,A)\to (C,A) \) which are homotopic to the identity. This is a version of what is known in the literature as \emph{\textbf{sliding boundaries}} (see Figure \ref{fig:sliding2}.) Sliding boundaries and minimizers have been studied in continuum mechanics for years, (see \cite{Podio-Guidugli}, for example.) and a definition for sliding boundaries suitable for geometric measure theory was introduced in \cite{david2014}, while \cite{delellisandmaggi} and \cite{ghiraldin} use a somewhat different one. \\

				\begin{figure}[h]
					\centering
					\includegraphics[width=.7\textwidth]{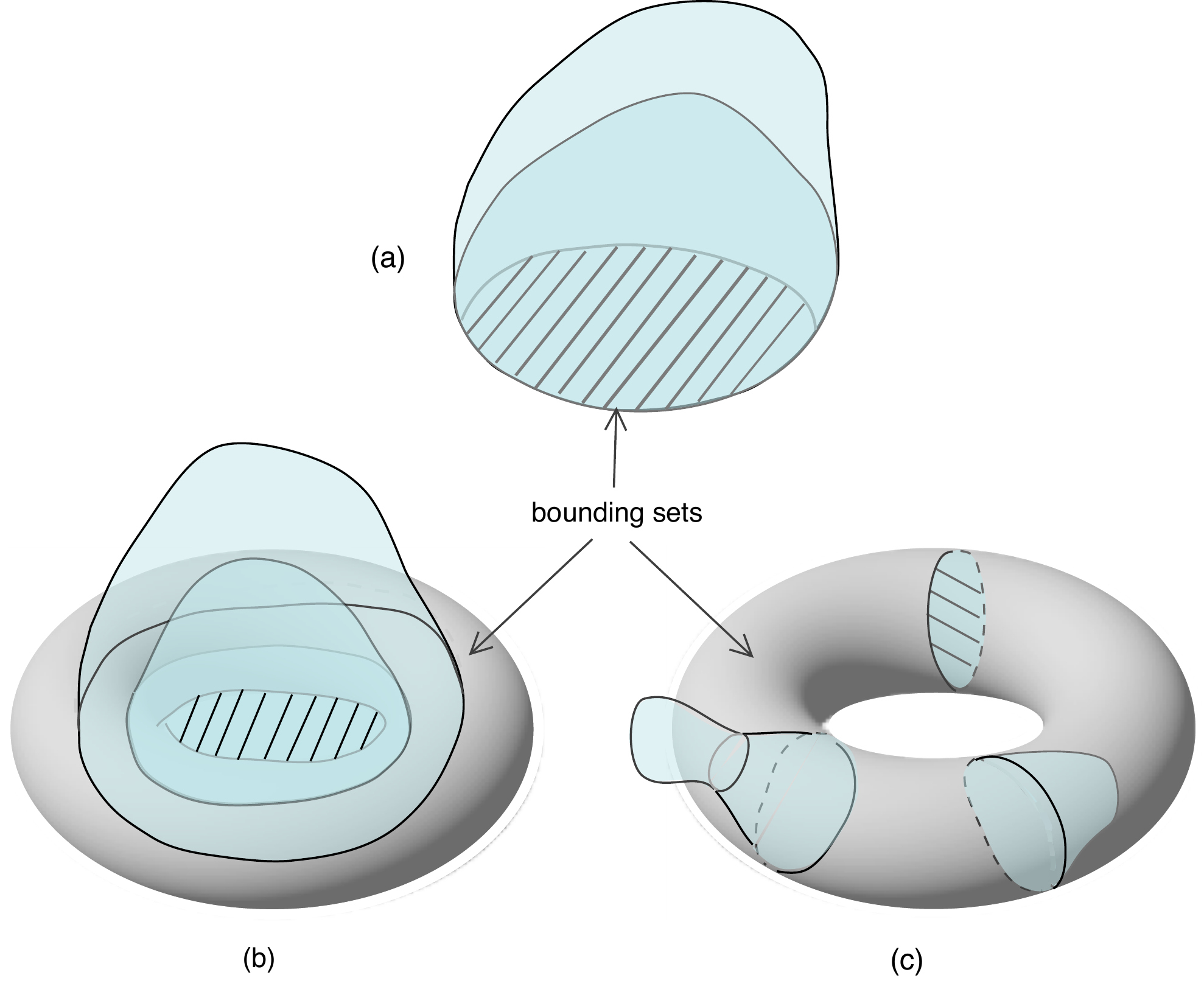}
					\caption{\textbf{Sliding boundaries}. In each of the three figures there are three competitors including a shaded one with minimal area. Figure (a) depicts the classical Plateau problem where the bounding set is a circle \( S \). The bounding set of both Figures (b) and (c) is a \( 2 \)-torus.}
					\label{fig:sliding2}
				\end{figure}
			\end{itemize}
		\end{examples}

	\subsection{More on ellipticity}
		\label{moreellipticty}
	
		Ellipticity of \( \F^m \) is implied in the following case: Suppose \( C \) is a Riemannian manifold. For \( p\in C \), let \( f_p: T^m C\to [a,b] \) denote the function \( (q,T)\mapsto f(p, \tilde{T}) \), where it is understood that
		\begin{enumerate}
			\item The function \( f_p \) is only defined for \( q \) in a geodesic neighborhood \( U_p \) of \( p \); and
			\item The \( m \)-plane \( \tilde{T} \) denotes the parallel transport of \( T \) along the unique minimizing geodesic from \( q \) to \( p \).
		\end{enumerate}
	
		For an \( m \)-rectifiable \( \H^m \) measurable set \( X\subset U_p \), define \( \F_p^m(X)=\int_X f_p(q,T_q X) d\H^m \). Suppose
		\begin{enumerate}
			\item\label{almgrenelliptic} For all \( p\in C \) there exists \( e(p)>0 \) such that
				\begin{equation}
					\label{eq:almgrenelliptic}
					\F_p^m(X)-\F_p^m(D)\geq e(p) (\H^m(X)-\H^m(D)),
				\end{equation}
				for all geodesic \( m \)-planes \( D\subset U_p \) and compact \( m \)-rectifiable sets \( X\subset U_p \) which contain \( \fr D \) and do not retract onto \( \fr D \); and
			\item The function \( f \) is equi-lower semicontinuous, in the sense that for each \( (p,T)\in T^mC \) and each \( \e>0 \) there exists \( \d>0 \), independent of \( T \), such that if \( d(p,q)<\d \), then \( f(p,T)\leq f(q,\tilde{T})+ \e \).
		\end{enumerate}
		Then \( \F^m \) is elliptic. Indeed, aside from a slight broadening of the collection of sets \( X \) which must satisfy \eqref{eq:almgrenelliptic}, the assumption \ref{almgrenelliptic} is precisely the ellipticity condition defined in \cite{almgrenannals} 1.2.

\section{Constructions}

Lemma \ref{lem:FF} is a slight generalization of \cite{davidsemmes} Proposition 3.1, which is a version of the Federer-Fleming projection theorem \cite{federerfleming} 5.5, first modified for sets in \cite{almgrenannals} 2.9 (see \cite{flatchains} for a much simpler proof.) Given a (closed) \( n \)-cube \( Q\subset \R^n \) and \( j \ge 0 \), let \( \D_j(Q) \) denote the collection of all \( n \)-cubes in the \( j \)-th dyadic subdivision of \( Q \). For \( 0 \le d \le n \) let \( \Delta_{j,d}(Q) \) denote the collection of the \( d \)-dimensional faces of the \( n \)-cubes in \( \D_j(Q) \) and let \( S_{j,d}(Q)\subset Q \) denote the set union of these faces.

\begin{lem}
	\label{lem:FF}
	Suppose \( E \) is a compact subset of \( Q \) such that \( \H^d(E) < \i \). For each \( j \ge 0 \), there exists a Lipschitz map \( \phi: \R^n\times [0,1] \to \R^n \) with the following properties:
	\begin{enumerate}
		\item\label{lem:ff:item:1} \( \phi_t = Id \) on \( Q^c\cup S_{j,d}(Q) \) for all \( t\in [0,1] \);
		\item\label{lem:ff:item:0} \( \phi_0 = Id \);
		\item\label{lem:ff:item:2} \( \phi_1(E) \subset S_{j,d}(Q) \cup \fr (Q) \);
		\item\label{lem:ff:item:3} \( \phi(R\times[0,1]) \subset R \) for each \( R \in \D_j(Q) \);
		\item\label{lem:ff:item:4} \( \H^d(\phi_1(E \cap R)) \le c_1 \H^d(E \cap R) \) for all \( R \in \D_j(Q) \) where \( c_1=c_1(n) \) depends only on \( n \);
		\item\label{lem:ff:item:5} \( \H^{d+1}(\phi(E\cap R\times [0,1]))\le c_1 2^{-j}\diam(Q) \H^d(E\cap R) \) for all \( R\in \D_j(Q) \);
		\item\label{lem:ff:item:6} If \( E \) is \( K \)-semiregular\footnote{See \cite{davidsemmes} Definition 3.29. Let \( K>0 \). A compact set \( E \) is \emph{\textbf{\( K \)-semiregular of dimension \( d \)}} if for \( x\in \R^n \) and \( 0<r\leq r' \), the set \( E(x,r') \) can be covered by a collection of at most \( K (r'/r)^d \) closed balls of radius \( r \).}, then the Lipschitz constant of \( \phi_1 \) depends only on \( n \) and \( K \). Furthermore, there exists a constant \( \d>0 \) depending only on \( n \) and \( K \) such that if \( F\subset Q\cap \cal{N}(E,\d 2^{-j}\diam(Q)) \), then \( \phi_1(F)\subset S_{j,d}(Q)\cup \fr(Q) \). 
	\end{enumerate}
\end{lem}

\begin{proof}
	The map \( \phi \) is the concatenation of the straight line homotopies between the maps \( \psi_m \) in \cite{davidsemmes} Lemma 3.10 roughly described as follows: Choose a small cubical grid and consider the union \( L \) of cubes in the grid that meet \( E \). Ignore all other cubes of the grid. For each grid cube \( R \subset L \), one can find a point \( q \in R\setminus E \), near the center of \( R \) for which there exists \( s > 0 \) with \( B(q,s) \subset R\setminus E \) (see \cite{davidsemmes} Lemma 3.10.) Use \( q \) to radially project \( E \cap R \) to \( \p R \). Repeat in each \( (n-1) \)-dimensional face of \( R \) and radially project to the \( (n-2) \)-skeleton of \( R \). Continue until the resulting image of \( E \cap R \) is contained in the \( d \)-skeleton of \( R \). Each projection determines a straight-line homotopy from the identity mapping to the projection.

	Parts \ref{lem:ff:item:1}-\ref{lem:ff:item:4} are \cite{davidsemmes} Proposition 3.1. Part \ref{lem:ff:item:5} is then apparent from \cite{reifenberg} Lemma 6 and \cite{davidsemmes} (3.14.) Part \ref{lem:ff:item:6} follows from \cite{davidsemmes} Lemma 3.31 and (3.33.)
\end{proof}

If \( Y\subset \R^n \) and \( p\in \R^n \), let \( C_p Y \) denote the union of the closed rays with one endpoint \( p \) and the other lying in \( Y \). We shall repeatedly make use of a ``cone construction'' in which a portion of a surface lying in a ball is replaced with the cone on the portion of the surface lying on the boundary of the ball. However, the resulting set \( C_p Y \) does not have very good estimates on its Hausdorff measure unless the set \( Y \) is polyhedral (see \cite{reifenberg} Lemmas 5 and 6.) Instead, we will first use a Federer-Fleming projection to push \( Y \) onto a cubical grid before coning, and this will yield nicer estimates.

The idea of deforming a surface to produce an approximation of the cone construction is due to \cite{delellisandmaggi}. We will take a Hausdorff limit of these deformations to produce a set which is contained in the outright cone. This will allow us to apply the cone construction to surfaces in a spanning collection and remain within the collection. One trivial yet important fact is the following:

\begin{lem}
	\label{lem:commutativity}
	If \( \{X_i\}_{i\in \N} \) is a sequence of compact subsets of \( \R^n \) and \( X_i \) converges to a compact set \( X \) in the Hausdorff metric, and if \( f: \R^n\to \R^n \) is continuous, then \( f(X_i) \) converges to \( f(X) \) in the Hausdorff metric.
\end{lem}
	
\begin{lem}
	\label{lem:deformcone}
	Suppose \( X\subset \R^n \) is compact, \( p\in \R^n \) and \( r>0 \). There exists a sequence of diffeomorphisms \( \xi_i \) of \( \R^n \) which are the identity outside \( B(p,r) \) and for which \( \xi_i(B(p,r))\subset B(p,r) \), such that \( \xi_i(X) \) converges in the Hausdorff metric to a compact set \( Y \) which is contained in \( X\setminus B(p,r) \cup C_p(x(p,r)) \).
\end{lem}

\begin{proof} 
	For \( t\in [0,1) \), let \( \psi_t:[0,\infty)\to [0,\infty) \) be a smooth, increasing function, such that 
	\begin{enumerate}
		\item \( \psi_t'(s)>0 \) for all \( s \);
		\item \( \psi_t(rt)< r(1-t) \);
		\item \( \psi_t(s)=s \) for all \( s\ge r \).
	\end{enumerate}

	Pick some sequence \( \{t_i\}_{i\in \N}\subset [0,1) \) with \( t_i\to 1 \) and let \( \xi_i(x) = p + \psi_{t_i}(|x-p|)\frac{x-p}{|x-p|} \) for \( x \in \R^n \). Then \( \xi_i \) is a diffeomorphism of \( \R^n \) satisfying \( \xi_i = Id \) on \( B(p,r)^c \) for all \( i\in \N \), and \( \lim_{i\to \i} \xi_i(x) = p \) for each \( x\in \cN(p,r) \).

	By taking a subsequence if necessary, we may assume without loss of generality that \( \{\xi_i(X)\}_{i\in \N} \) converges in the Hausdorff metric to a compact set \( Y \). By construction, this set \( Y \) is contained in \( X\setminus B(p,r) \cup C_p(x(p,r)) \).
\end{proof}

\begin{lem}
	\label{cor:deformation}
	Suppose \( (C,X,A) \) is a compact triple, \( C \) is a uniform localizable Lipschitz retract, and \( X\in \Span(C,A) \). Let \( p\in C\setminus A \) and suppose \( r>0 \) is chosen smaller than the retraction radius of \( C \) at \( p \), and so that \( B(p,\sqrt{n}r)\cap A=\emptyset \), \( x(p,r) \) is \( (m-1) \)-rectifiable, and \( \H^{m-1}(x(p,r))<\i \). Then for every \( \e>0 \) there exist compact sets \( P_\e \subset C(p,r) \) and \( T_\e \subset C(p,r) \) such that
	\begin{enumerate}
		\item\label{cor:deformation:item:2} \( T_\e \subset \cN(x(p,r), \e r) \);
		\item\label{cor:deformation:item:3} \( \H^m(T_\e) \le c_1 \e r \H^{m-1}(x(p,r)) \);
		\item\label{cor:deformation:item:4} \( \H^m(P_\e) \le \g r^m \) where \( 0<\g <\i \) depends on \( n, C \) and \( \e \);
		\item\label{cor:deformation:item:1} \( (P_\e \cup T_\e \cup (X \setminus B(p,r)))^\dagger \) is an element of \( \Span(C,A) \).
	\end{enumerate}
\end{lem}

\begin{proof}
	The set \( X\setminus X(p,r) \cup C_p(x(p,r)) \) contains a set \( Y \) which is a Hausdorff limit of deformations of \( X \) by diffeomorphisms of \( \R^n \) by Lemma \ref{lem:deformcone}. We shall deform \( Y \) using a modification of \ref{lem:FF} and construct for each \( 0<\d<1 \) a Lipschitz deformation of \( B(p,r) \) that maps each sphere \( \p B(p,s) \) to itself for \( 0 < s \le r \), maps radial rays to radial rays on \( B(p,(1-\d)r) \) and is the identity on \( \p B(p,r) \). 
	
	Let \( \Pi_r: \R^n\setminus \{p\} \to \p B(p,r) \) and \( \Pi_{r,s}: \p B(p,r) \to \p B(p,s) \) denote radial projections. 
	 
	Let \( Q \) be an \( n \)-cube of side length \( 2r \) centered at \( p \) and apply Lemma \ref{lem:FF} to \( d = m-1 \) and \( E = y(p,r) \) and for a fixed number \( j \) of subdivisions of \( Q \), to be determined in a moment. Obtain the Federer-Fleming map \( \phi:\R^n \times [0,1] \to \R^n \) from Lemma \ref{lem:FF} and let \( \tilde{\phi}_t = \Pi_r \circ \phi_t \). Since \( y(p,r) \) is \( (m-1) \)-rectifiable, so is \( \tilde{\phi}_t(y(p,r)) \). So, for each \( t\in [0,1] \), the map \( \tilde{\phi}_t \) restricts to a map from \( B(p,r) \) to itself. At \( t=0 \) the map is the identity and at \( t=1 \), it sends \( y(p,r) \) to \( \Pi_r(\phi_1(y(p,r))) \).

	Using this homotopy, we shall define a Lipschitz map \( \psi: \R^n \to \R^n \) so that \( \psi \) sends each sphere \( \p B(p,s) \) to itself for \( 0 < s \le r \). Suppose \( (1 - \d)r \le s \le r \).

	Let \[ \psi\lfloor_{\p B(p,s)} = \Pi_{r,s} \circ \tilde{\phi}_{(r-s)/(\d r)} \circ \Pi_{r,s}^{-1}. \] Extend \( \psi \) to \( B(p, r - \d r) \) in the unique way such that each ray from \( p \) to \( q \in \p B(p, r - \d r) \) is mapped to the ray from \( p \) to \( \psi(q) \) and so that \( \psi(B(p,s))\subset B(p,s) \) for each \( 0\leq s<(1-\d)r \). Finally, extend \( \psi \) to the identity on \( B(p,r)^c \).

	Let \( \pi_{p,r}:\R^n\to C(p,r) \) denote the Lipschitz retraction given by Definition \ref{def:LNR}. It follows from Lemma \ref{lem:FF} \ref{lem:ff:item:5} that we may choose \( j \) large enough and \( 0<\d<1 \) small enough so that \[ \tilde{T_\e} := \pi_{p,r} (\psi(Y(p,r)\setminus \cN(p,(1-\d)r))) \] satisfies \ref{cor:deformation:item:3} and \ref{cor:deformation:item:2}. The set \( T_\e \) will be defined as a subset of \( \tilde{T_\e} \).
 	
	Likewise, let \[ \tilde{P_\e} = \pi_{p,r} (\psi(Y(p, (1-\d)r))). \] We will define \( P_\e \) as a subset of \( \tilde{P_\e} \), so let us establish \ref{cor:deformation:item:4}: Let \( N_\e \) be an upper bound on the number of \( (m-1) \)-dimensional faces of a cubical grid of side length \( 2^{-j(\e)} \) within \( \e \) of \( \fr B(0,1) \). Since \( \phi_1(x(p,r)) \subset \cN(x(p,r),\e r)\cap S_{j,m-1}(Q) \), it follows from \cite{reifenberg} Lemmas 2 and 5 that
	\begin{align*}
		\H^m(P_\e) &\leq \frac{\kappa^m r\sqrt{n} N_\e}{m} \left(\frac{r}{2^{j+1}}\right)^{m-1}\\ 
		&\leq \kappa^m \sqrt{n} N_\e \left(\frac{\e}{2\sqrt{n}}\right)^{m-1} r^m.
	\end{align*}
	
	We know that for each \( i \), the map \( \pi_{p,r}\circ \psi \circ \xi_i:C\to C \) is Lipschitz and homotopic to the identity relative to \( A \) (indeed, the map fixes the complement of a ball which misses \( A \)) and so by Axiom \eqref{lip}, we know that \( (\pi_{p,r}\circ \psi \circ \xi_i (X))^\dagger\in \Span(C,A) \).  By Lemma \ref{lem:commutativity}, we know that \( \pi_{p,r}\circ \psi \circ \xi_i (X) \) converges in the Hausdorff metric to \( \tilde{Z}=\tilde{P_\e} \cup \tilde{T_\e} \cup (X \setminus B(p,r)) \), which by construction is \( m \)-rectifiable away from \( A \) and satisfies \( \H^m(Z\setminus A)<\i \). By taking a subsequence if necessary, \( (\pi_{p,r}\circ \psi \circ \xi_i (X))^\dagger \) converges in the Hausdorff metric to a subset \( Z \) of \( \tilde{Z} \) which contains \( X\setminus B(p,r) \). Let \( P_\e=\tilde{P_\e}\cap Z \) and let \( T_\e=\tilde{T_\e}\cap Z \). Thus, by Axiom \eqref{lim}, \eqref{cor:deformation:item:1} holds.
	
	Now if \( C \) is a Riemannian manifold and Axiom \eqref{lip}' holds in place of axiom \eqref{lip}, then we proceed as above in a coordinate chart and omit the map \( \pi_{p,r} \). That \( \psi \) is uniformly approximable by diffeomorphisms isotopic to the identity is shown in Lemma \ref{lem:ffapprox}.
\end{proof}

\begin{lem}
	\label{lem:ffapprox}
	Suppose \( E \) is a compact subset of \( Q \) such that \( \H^d(E) < \i \), \( U\subset E \) is \( \H^d \) measurable and purely \( d \)-unrectifiable, and \( \bar{U}\subset \ring{Q} \). Then for \( 0\le j<\i \) large enough, there exists a map \( \phi \) satisfying Lemma \ref{lem:FF} \ref{lem:ff:item:0}-\ref{lem:ff:item:6} such that
	\begin{equation}
		\label{lem:ffapprox:1}
		\H^d(\phi_1(U))=0
	\end{equation}
	and 
	\begin{equation}
		\label{lem:ffapprox:2}
		\phi_t = Id \text{ on } Q^c \text{ for all } t\in [0,1].
	\end{equation}
\end{lem}

\begin{proof}
	We construct the map \( \phi \) as follows. Choose \( j \) large enough so that \( U \) is contained entirely within the subset of cubes in \( \D_j(Q) \) which do not have any faces on the boundary of \( Q \). Within those cubes, approximate the maps \( \psi_k \), \( d \leq k < n \) defined in the proof of \cite{davidsemmes} Proposition 3.1 by diffeomorphisms \( \tilde{\psi}_k \). For each \( d\leq k< n \), the image of \( U \) by the map \( \tau_k \equiv \tilde{\psi}_k\circ\cdots\circ \tilde{\psi}_{n-1} \) remains purely \( d \)-unrectifiable and \( \tau_k(E) \) will be contained in an open \( \e \)-neighborhood of \( \Delta_{j,k}(Q) \). 
	
	The map \( \phi_1 \) is the composition \( \phi_1\equiv\rho\circ \theta\circ\tau_d \), where the maps \( \rho \) and \( \theta \) are defined below. The map \( \phi \) is defined as in Lemma \ref{lem:FF} to be the concatenation of the straight-line homotopies between the maps \( \psi_n, \tilde{\psi}_{n-1}, \dots, \tilde{\psi}_d, \theta \), and \( \rho \).
	
	By the Besicovitch-Federer projection theorem, for each \( 0<\d<\e \) and each \( d \)-face \( F\in \Delta_{j,d}(Q) \) there exists an affine \( d \)-plane \( \tilde{F} \) such that \( F\subset \cal{N}(\tilde{F},\d) \), and such that the image of \( \phi_1(U) \) by the orthogonal projection \( \Theta_{\tilde{F}}: \R^n \to \tilde{F} \) has zero Hausdorff \( d \)-measure. 
	
	Let \( \zeta>2\e \) and define \( \theta \) on \( \cal{N}(F, \e)\setminus \cal{N}(\p F, \zeta) \) to be the map \( \Theta_F\circ \Theta_{\tilde{F}} \), where \( \Theta_F \) denotes orthogonal projection onto the affine \( d \)-plane defined by \( F \). We may extend \( \theta \) as a Lipschitz map on \( \cal{N}(\Delta_{j,d}(Q), \e)\cap B(\Delta_{j,d-1}(Q),\zeta) \) so that
	\begin{equation}
		\label{lem:ffapprox:a}
		\theta (\cal{N}(\Delta_{j,d}(Q), \e)\cap B(\Delta_{j,d-1}(Q),\zeta)) \subset \Delta_{j,d}(Q)\cap \cal{N}(\Delta_{j,d-1}(Q),2\zeta),
	\end{equation}
	and so that the Lipschitz constant of \( \theta \) depends only on \( n \). Finally, extend \( \theta \) to \( \R^n \) as a Lipschitz map.
	
	The map \( \rho \) is defined on each face \( F\in \Delta_{j,d}(Q) \) as follows: Let \( q \) be the center point of \( F \) and let \( \chi_{q,\zeta}(x) \) denote the point \[ \frac{1}{1-2\sqrt{d}\zeta}(x-q)+q. \] For \( x\in F\setminus \{q\} \), let \( \omega(x) \) denote the point on \( \p F \) and the ray passing through \( x \) and ending at \( q \). For \( x\in F \), let
	\begin{align*}
		\rho(x)=\left\{\begin{array}{lr}
							\chi_{q,\zeta}(x), & \text{if } \chi_{q,\zeta}(x)\in F\\
							\omega(x), & \text{otherwise.}
							\end{array}\right\}
	\end{align*}
	
	Then \( \rho \) is Lipschitz with Lipschitz constant close to 1 (controlled by \( \zeta \).) Finally, extend \( \rho \) to \( \R^n \), with proportional Lipschitz constant. 
	
	Note that by \eqref{lem:ffapprox:a}, \( \H^d \) almost all \( \theta(U) \) is contained in \( \Delta_{j,d}(Q)\cap \cal{N}(\Delta_{j,d-1}(Q), 2\zeta) \), and this region is collapsed onto \( \Delta_{j,d-1}(Q) \) by \( \rho \). Thus, \eqref{lem:ffapprox:1} holds. It is apparent that \eqref{lem:ffapprox:2} holds, as well as Lemma \ref{lem:FF} \ref{lem:ff:item:0}-\ref{lem:ff:item:3}. To see \ref{lem:ff:item:4} and \ref{lem:ff:item:5}, it is enough to observe that \cite{davidsemmes} (3.20) still holds for the modified maps \( \tilde{\psi}_k \). Finally, \ref{lem:ff:item:6} holds since \cite{davidsemmes} (3.33) still holds for the modified maps.
\end{proof}

\begin{lem}[Upper bounds on density ratios]
	\label{lem:upperdensity}
	Suppose \( Y_k\subset \R^n \) for \( k \ge 1\). Fix \( p \in \R^n \), \( 0 < r < R <\i \) and \( 0\leq \eta\leq \i \). Suppose \[ \limsup_{k\to\i} \frac{ \H^m(Y_k(p,R))}{R^m} \le \eta \] and that for each \( \d>0 \) there exists \( M_\d \) such that if \( k \geq M_\d \), then
	\begin{equation}
		\label{eq:xkps}
		\H^m(Y_k(p,s)) < \frac{s}{m} \frac{d}{ds}\H^m(Y_k(p,s)) + \d
	\end{equation}
	for almost every \( s\in [r, R) \) satisfying \( \H^m(Y_k(p,s))/s^m \ge \eta \). Then 
	\[ \limsup_{k\to\i} \frac{ \H^m(Y_k(p,r))}{r^m} \le \eta. \]
\end{lem}

\begin{proof}
	If \( \eta=\i \) or \( \eta=0 \) there is nothing to prove, so let us assume \( 0<\eta<\i \). Suppose \[ \limsup_{k\to\i} \frac{ \H^m(Y_k(p,r))}{r^m} > \eta. \] Let \( \e > 0 \) and 
	\begin{equation}
		\label{eq:defdelt}
		0 < \d < \min\left\{ \e, \left(\limsup_{k \to \i} \frac{\H^m(Y_k(p,r))}{r^m} - \eta \right) \right\} \frac{r^{m+1}}{mR}.
	\end{equation}

	Let \( k_i \to \i \) with \( k_i>M_\d \),
	\begin{equation}
		\label{eq:defdelt1}
		\left(\frac{\H^m(Y_{k_i}(p,r))}{r^m} - \eta \right) \frac{r^{m+1}}{mR}> \d,
	\end{equation}
 
	\begin{equation}
		\label{eq:defdelt2}
		\H^m(Y_{k_i}(p,r)) \to \limsup_{k \to \i} \H^m(Y_k(p,r)),
	\end{equation}
	and
	\[
		\frac{\H^m(Y_{k_i}(p,r))}{r^m} > \eta.
	\]

	Let \( J_i = [r,r_i) \) be the largest half-open interval with right endpoint \( r_i\in (r,R] \) such that \( \H^m(Y_{k_i}(p,t))/t^m \ge \eta \) for almost every \( t \in J_i \). By \eqref{eq:defdelt1},
	\begin{equation}
		\label{eq:stuff}
		\frac{\H^m(Y_{k_i}(p,r))}{r^m} > \eta + \frac{m}{r^{m+1}} \d(r_i-r).
	\end{equation}
	
	By \eqref{eq:xkps}, we have 
	\begin{align*}
		 \frac{d}{dt} \frac{\H^m(Y_{k_i}(p,t))}{t^m} = \frac{t \frac{d}{dt}(\H^m(Y_{k_i}(p,t))) - m \H^m(Y_{k_i}(p,t)))}{t^{m+1}} > -\frac{m \d} {t^{m+1}} \geq - \frac{m\d}{r^{m+1}}
	\end{align*}
	for almost every \( t\in J_i \).

	Integrating yields,
	\begin{equation}
		\label{eq:stuff2}
		-\frac {m\d(r_i-r)}{r^{m+1}} < \int_r^{r_i} \frac{d}{dt}\left(\frac{\H^m(Y_{k_i}(p,t))}{t^m}\right) dt \le \frac{\H^m(Y_{k_i}(p,r_i))}{r_i^m} - \frac{\H^m(Y_{k_i}(p,r))}{r^m}
	\end{equation}
	 
	Combining \eqref{eq:stuff} and \eqref{eq:stuff2} yields
	\begin{equation*}
		\frac{\H^m(Y_{k_i}(p,r_i))}{r_i^m}> \eta.
	\end{equation*}

	It follows that \( r_i = R \). Since \( \d \leq \e r^{m+1}/(m R) \), \eqref{eq:stuff2} implies
	\begin{equation}
		\frac{\H^m(Y_{k_i}(p,r))}{r^m} < \frac{\H^m(Y_{k_i}(p,R))}{R^m} + \e.
	\end{equation}
	
	Letting \( i\to \i \) and then \( \e\to 0 \) we deduce from \eqref{eq:defdelt2} that \[ \limsup_{k \to \i}\frac{\H^m(Y_{k_i}(p,r))}{r^m} \le \eta, \] a contradiction.
\end{proof}

For \( E\in \Gr(m,n) \), \( p\in \R^n \) and \( 0<\e<1 \), let \( \cal{C}(p,E,\e) = \{q \in \R^n: d_H(q-p, E) < \e d_H(p,q) \} \).

\begin{lem}
	\label{lem:approx}
	Let \( X\subset \R^n \), \( p\in X \), and suppose there exists \( s > 0 \) such that \[ \inf\left\{ \frac{\H^m(X(q,r))}{r^m}: q\in X, B(q,r)\subset B(p,s) \right\}>0. \] If \( E\in \Gr(m,n) \) is an approximate tangent \( m \)-plane for \( X \) at \( p \), then for every \( 0<\e<1 \) there exists \( r>0 \) such that \( X(p,r)\setminus \overline{\cal{C}(p,E,\e)}=\emptyset \).
\end{lem}
		
\begin{proof} 
	If the result is false, then for some \( 0<\e<1 \) there exist sequences \( r_i \to 0 \) and \( q_i \in X(p,r_i) \setminus \overline{\cal{C}(p,E,\e)} \). Let \( s_i = 2d_H(q_i,p) \). Then \[ B(q_i,\e s_i/4) \subset B(p,s_i) \setminus \cal{C}(p,E,\e/4) \] and thus
	\[
		\frac{\H^m(X(p,s_i)\setminus \cal{C}(p,E,\e/4))}{s_i^m } \ge \frac{\H^m(X(q_i,\e s_i/4))}{s_i^m},
	\]
the right hand side of which for large enough \( i \) is bounded below by
\[
	 (\e/4)^m \inf \left\{\frac{\H^m(X(q,r))}{r^m}: q\in X, B(q,r)\subset B(p,s) \right\} > 0,
\] 
a contradiction.
\end{proof} 

\section{Minimizing sequences}\label{minseq}
 
	\begin{defn}\label{def:rregular}
		We say that a sequence \( \{X_k\}_{k\in \N} \) is \emph{\textbf{Reifenberg regular in \( A^c \)}} if there exist \( 0<\mathbf{c}<\i \) and \( 0<\mathbf{R}\leq \i \) such that if \( k\geq 1 \), \( p\in X_k \), \( 2^{-k}<r< \mathbf{R} \) and \( B(p,r) \) is disjoint from \( A \), then \[ \frac{\H^m(X_k(p,r))}{r^m}\geq \mathbf{c}. \]
	\end{defn}
		
	We shall make the following assumptions for the remainder of this paper: Assume \( C \) is a compact uniform localizable Lipschitz neighborhood retract, that \( A\subset C \) is closed and that \( \F^m \) is elliptic. Fix an embedding \( C\hookrightarrow \R^n \) as a uniform localizable Lipschitz neighborhood retract and let \( \pi: U\to C \) be a Lipschitz retraction of an open neighborhood \( U \) of \( C \). Suppose \( \Span(C,A) \) is a non-empty spanning collection and suppose \( \{X_k\}_{k\in \N} \subset \Span(C,A)\) satisfies:

	\begin{enumerate}
		\item \( \F^m(X_k\setminus A) \to \frak{m} := \inf\{\F^m(Y\setminus A):Y \in \Span(C,A) \} \); 
		\item The measures \( \F^m\lfloor_{X_k\setminus A} \) converge weakly to a finite Borel measure \( \mu_0 \);
		\item \( X_k\to X_0 \equiv \supp(\mu_0)\cup A \) in the Hausdorff metric;
		\item \( \{X_k\} \) is Reifenberg regular.
	\end{enumerate}
	
	It was shown in \cite{lipschitz} that there exists such a sequence \( \{X_k\} \), as long as \( \Span(C,A) \) is defined using an algebraic spanning condition with \( L_3=L_4=\emptyset \). For the existence of such a sequence to hold in the full generality of axiomatic spanning conditions, it suffices to show\footnote{See \cite{lipschitz} Lemma 4.2.3 for details.} that if \( X\in \Span(C,A) \), \( p\in C\setminus A \), and \( r>0 \) is small enough so that \( B(p,r)\cap A=\emptyset \), then \( X\setminus B(p,r)\cup \pi(Y) \) contains an element of \( \Span(C,A) \), where \( Y \) is the set generated in \cite{reifenberg} Lemma 8 from the set \( X\cap \fr(p,r) \). That this is indeed the case is shown in \cite{isoperimetric}. So, let us assume we have such a sequence \( \{X_k\}. \) By \cite{lipschitz} Corollary 4.3.5, 
	\begin{equation}
		\label{eq:lowerbound}
		\frac{\mu_0(B(p,r))}{r^m}\geq a\mathbf{c}>0
	\end{equation}
	for all \( p\in X_0\setminus A \) and \( 0<r<\min\{ d_H(p,A),\mathbf{R} \} \). In particular, there is a uniform lower bound on lower density: \( {\Theta_*}^m(\mu_0,p)\geq a\mathbf{c}\a_m^{-1}>0 \). We now show there is an upper bound for \( \frac{\mu_0(B(p,r))}{r^m} \), uniform away from \( A \).
	
	\begin{defn}
		\label{def:dppp}
		For \( p \in A^c \) let \( d_p=\min\{ d_H(p,A\cup U^c),\mathbf{R}, \xi_p \} \). Let \( D_p \) be the subset of \( (0,d_p) \) consisting of numbers \( r \) such that the following conditions hold for all \( s \in \{ qr: q\in (0,1]\cap \mathbb{Q} \} \) and \( k\geq 1 \):
		\begin{enumerate}
			\item\label{def:dppp:item:1} \( \mu_0(\fr B(p,s)) = 0 \),
			\item\label{def:dppp:item:0} \( x_k(p,s) \) is \( (m-1) \)-rectifiable,
			\item\label{def:dppp:item:2} \( \H^{m-1}(x_k(p,s))<\i \),
			\item\label{def:dppp:item:3} \( s\mapsto \H^m(X_k(p,s)) \) is differentiable at \( s \), and
			\item\label{def:dppp:item:4} \[ \lim_{h\to 0}\frac{1}{h}\int_s^{s+h}\H^{m-1}(x_k(p,t))dt = \H^{m-1}(x_k(p,s)). \]
			\end{enumerate}
	\end{defn}

	\begin{lem}
		\label{lem:prime}
		\( D_p \) is a full Lebesgue measure subset of \( (0,d_p) \).
	\end{lem}

	\begin{proof} 
		Part \ref{def:dppp:item:0} determines a full Lebesgue measure set since \( X_k \) is \( m \)-rectifiable. Part \ref{def:dppp:item:2} follows from \cite{lipschitz} Lemma 2.0.1. Part \ref{def:dppp:item:3} follows since \( \H^m(X_k(p,s)) \) is monotone non-decreasing. Part \ref{def:dppp:item:4} follows from the Lebesgue Differentiation Theorem.
	\end{proof}

	\begin{lem}
		\label{lem:ldt}
		If \( r \in D_p \), then \( \H^{m-1}(x_k(p,r)) \le \frac{d}{dr}\H^m(X_k(p,r)) \) for all \( k\geq 1 \). 
	\end{lem} 

	\begin{proof}
		By \cite{lipschitz} Lemma 2.0.1 and Definition \ref{def:dppp},
		\begin{align*}
			 \H^{m-1}(x_k(p,r)) &=\lim_{h\to 0} \frac{\int_r^{r+h} \H^{m-1}(x_k(p,t))\,dt}{h}\\ 
			 					& \le \lim_{h\to 0} \frac{\H^m(X_k(p,r+h)\setminus X_k(p,r))}{h}\\
								& = \frac{d}{dr}\H^m(X_k(p,r))
		\end{align*}
	\end{proof}
		 
	Let \( M<\i \) be an upper bound for \( \{\F^m(X_k\setminus A) : k\geq 1\} \).
	
	\begin{lem}
		\label{lem:min}
		Let \( p \in X_0 \setminus A \) and \( 0<r\leq d_p \). For each \( \d > 0 \) there exists \( N_{p,r,\d} > 1 \) such that if \( k\geq N_{p,r,\d} \), and \( Y_k\in \Span(C,A) \) satisfies
		\begin{equation}
			\label{eq:lemmin1}
			Y_k=(X_k\cap \cN(p,r)^c) \cup Z_k
		\end{equation}
		for some \( \H^m \) measurable \( m \)-rectifiable set \( Z_k\subset C \), then
		\begin{equation}
			\label{eq:eqlemmin}
			\F^m(X_k \cap \cN(p,r)) < (1+\d)\F^m(Z_k \setminus A).
		\end{equation}
		
		If, in addition \( 0<s<r \), \( \d\leq 1 \) and \( Z_k\setminus Z_k(p,s)=X_k\cap \cN(p,r)\setminus X_k(p,s) \), then
		\begin{equation}
			\label{eq:HmXk}
			\F^m(X_k(p,s)) \le 2 \F^m(Y_k(p,s)) + \d M.
		\end{equation}

	\end{lem}
	
	\begin{proof}
		If \eqref{eq:eqlemmin} fails, there exist \( k_i \to \i \) and \( Y_{k_i}\in \Span(C,A) \) satisfying \eqref{eq:lemmin1} such that \[ \F^m(Z_{k_i}\setminus A) \leq \frac{\F^m(X_{k_i}\cap \cN(p,r)}{(1+\d)}. \] It follows from \cite{lipschitz} Proposition 4.3.2 that
		\begin{align*}
			\liminf_{i\to \i}\{\F^m(Y_{k_i}\setminus A) \} &\le \liminf_{i\to \i}\{\F^m(Z_{k_i}\setminus A) + \F^m((X_{k_i}\cap \cN(p,r)^c)\setminus A) \} \\
			&\le \liminf_{i\to \i}\{\F^m(X_{k_i}\setminus A) - \frac{\d}{1+\d}\,\F^m(X_{k_i}\cap \cN(p,r)) \} \\
			&\le \liminf_{i\to \i}\{\F^m(X_{k_i}\setminus A) - \frac{a \d}{1+\d}\, \H^m(X_{k_i}(p,r/2)) \} \\
			&\le \liminf_{i\to \i}\{\F^m(X_{k_i}\setminus A)\} - \frac{a \d}{1+\d}\, \mathbf{c} (r/2)^m \\
			&< \frak{m},
		\end{align*}
		a contradiction.

		By \eqref{eq:eqlemmin},
		\begin{equation}
		 	\F^m(X_k(p,s)) + \F^m(X_k\cap \cN(p,r)\setminus X_k(p,s)) \le (1+\d)\F^m(Z_k(p,s)) +(1+\d)\F^m(Z_k\setminus Z_k(p,s)),
		\end{equation}
		and thus
		\begin{equation}
			\F^m(X_k(p,s)) \le (1+\d)\F^m(Z_k(p,s)) + \d \F^m(Z_k\setminus Z_k(p,s)) \le (1+\d)\F^m(Z_k(p,s))+ \d \F^m(X_k\setminus A).
		\end{equation}
	\end{proof}

	Let \( K = 2b/a \) and \( \e_0 = 1/(2c_1mK) \). Let \( \g_0 \) denote the constant \( \text{``\(\gamma\)''} \) produced from Lemma \ref{cor:deformation} corresponding to \( n \), \( C \) and \( \e_0 \).

	\begin{lem}
		\label{lem:muk}
		Let \( p \in X_0 \setminus A \) and \( 0<r\leq d_p \). If \( s\in D_p\cap (0,r) \), \( \d\leq 1 \), and \( k \ge N_{p,r,\d} \), then
		\begin{equation}
			\label{eq:muk1}
			\H^m(X_k(p,s)) \le K\g_0s^m + \frac{s}{2m} \frac{d}{ds}\H^m(X_k(p,s)) + \d M/a.
		\end{equation}
		If in addition
		\begin{equation}
			\label{eq:muk2}
			\frac{\H^m(X_k(p,s))}{s^m} \ge 2K\g_0,
		\end{equation}
		then
		\begin{equation}
			\label{eq:theimportanteq}
			\H^m(X_k(p,s)) \le \frac{s}{m} \frac{d}{ds}\H^m(X_k(p,s)) + 2\d M/a.
		\end{equation}
	\end{lem}
	
	\begin{proof}
		Let \( \hat{X}_k \) denote the set \( P_k \cup T_k \cup (X_k \setminus B(p,r)) \), where \( P_k \) and \( T_k \) are the sets produced from Lemma \ref{cor:deformation} applied to \( \text{``\( X \)''}=X_k \), \( \text{``\( r \)''}=s \) and \( \text{``\( \e \)''}=\e_0 \). Lemma \ref{cor:deformation} and Lemmas \ref{lem:min} and \ref{lem:ldt} then yield
		\begin{align}
			\label{eq:min}
			\H^m(X_k(p,s)) &\le K \H^m(\hat{X}_k(p,s)) + \d M/a\\
			&\le K(\g_0 s^m + c_1 \e_0 s \H^{m-1}(x(p,s))) + \d M/a\\
			&\le K\g_0s^m + \frac{s}{2m} \frac{d}{ds}\H^m(X_k(p,s)) + \d M/a.
		\end{align}
		The second assertion follows from algebraic manipulation of \eqref{eq:muk1}. 
	\end{proof} 

	Let 
	\begin{equation}
		\label{eq:c(p)}
		c(p)=\max \{ M/(a d_p^m), 2K\g_0 \}.
	\end{equation}
	
	\begin{thm}
		\label{thm:upperdensity}
		Let \( p \in X_0 \setminus A \) and \( 0<r< d_p \). Then \[ \limsup_{k \to \i}\frac{\H^m(X_k(p,r))}{r^m}\leq c(p). \] In particular,
		\begin{equation}
			\label{eq:upperbound}
			\frac{\mu_0(B(p,r)) }{r^m}\leq b c(p) 
		\end{equation}
	\end{thm}
	
	\begin{proof}
 		Let \( r<R<d_p \). By Lemmas \ref{lem:muk} and \ref{lem:prime}, we may apply Lemma \ref{lem:upperdensity} to prove that \[ \limsup_{k \to \i}\frac{\H^m(X_k(p,r))}{r^m}\leq \max \left\{ \frac{M}{a R^m}, 2K\g_0 \right\}, \] using the inputs \( \text{``\( Y_k \)''}=X_k \) and \( \text{``\( \eta \)''}= \max \left\{ \frac{M}{a R^m}, 2K\g_0 \right \} \). Take \( R\to d_p \).
		
		Now \eqref{eq:upperbound} follows from the Portmanteau theorem, since \( D_p \) is dense in \( (0,d_p) \).
	\end{proof}

	For \( 0<r<d_p \), let \( c_r(p)=\sup_{ q\in X_0(p,r)} \{ c(q)\} <\i \).
	\begin{cor}
		\label{cor:densitybound}
		If \( p \in X_0\setminus A \) and \( 0<r<d_p \), then \[ 0 < a\mathbf{c}/\a_m \le {\Theta_*}^m(\mu_0,q) \le {\Theta^*}^m(\mu_0,q) \le bc_r(p)/\a_m < \i \] for all \( q \in X_0(p,r) \).
	\end{cor}
	 							
	\begin{proof}
		The lower bound is due to \cite{lipschitz} Corollary 4.3.5. The upper bound follows from Theorem \ref{thm:upperdensity}: We have by Lemma \ref{lem:prime} and the Portmanteau theorem,
		
		\begin{align*}
			{\Theta^*}^m(\mu_0,q) &= \limsup_{t \to 0, t\in D_q} \frac{\mu_0(B(q,t))}{\a_m t^m}\\
			&= \limsup_{t \to 0, t\in D_q} \lim_{k \to \i} \frac{\F^m\lfloor_{X_k\setminus A}B(q,t)}{\a_m t^m}\\
			&\le \limsup_{t \to 0, t\in D_q} \limsup_{k \to \i} b\frac{\H^m(X_k(q,t)}{\a_m t^m}\\
			&\le \frac{b c(q)}{\a_m }\le \frac{b c_r(p)}{\a_m} < \i.
		\end{align*}
	\end{proof}

	Using \cite{mattila} 6.9, we deduce the following corollary:
			\begin{cor}
		\label{cor:matapplied}
		If \( p \in X_0\setminus A \) and \( 0<r<d_p \), then \[ a \mathbf{c} \H^m(X_0(p,r)) \le \a_m\mu_0(B(p,r)) \le b c_r(p) 2^m \H^m(X_0(p,r)). \]
	\end{cor}
	
	\begin{cor}
		\label{cor:semiregular}
		\( X_0\setminus \cal{N}_\e(A) \) is semiregular for every \( \e>0 \).
	\end{cor}
	
	\begin{proof}
		Let \( Y=X_0\setminus \cal{N}_\e(A) \). We show first that there exists a constant \( C \) such that if \( x\in \R^n \) and \( 0<r\leq R<\e'\equiv \min{\mathbf{R}/4,\e/4} \), then \( Y(x,R) \) can be covered by \( C (R/r)^m \) balls of radius \( r \). Indeed, suppose \( \{p_i\}_{i\in I} \) is a maximal family of points in \( Y(x,R) \) which are of distance \( \geq r \) from each other. Then, by \eqref{eq:lowerbound} and Theorem \ref{thm:upperdensity},
		\begin{align*}
			r^m|I|= 2^m \sum_{i\in I}(r/2)^m &\leq 2^m \sum_{i\in I} \frac{\mu_0(B(p_i,r/2))}{a\mathbf{c}}\\
			&\leq 2^m \frac{\mu_0(\cup_{i\in I}B(p_i,r/2))}{a\mathbf{c}}\\
			&\leq \frac{2^m}{a \mathbf{c}} \mu_0(B(x,2R))\\
			&\leq \frac{2^m}{a \mathbf{c}} \mu_0(B(p,4R))\\
			&\leq \frac{2^{3m}}{a \mathbf{c}} b c(p) R^m\\
			&\leq \frac{2^{3m}}{a \mathbf{c}} b \max \{ M/(a (4\e')^m), 2K\g_0 \} R^m,
		\end{align*}
		where \( p \) is any point in \( \{p_i\}_{i\in I} \). The last inequality is due to \eqref{eq:c(p)} and the fact that \( 4\e' \) is a lower bound for \( d_p \) for \( p\in Y \). 
		
		The general case follows from the finiteness of \( \mu_0 \) and compactness of \( Y \). Indeed, if \( r<\e'\leq R \), then it is enough to show that \( Y \) can be covered by \( C r^{-m} \) balls of radius \( r \). The proof is the same as the first case, replacing \( \mu_0(B(p,4R)) \) with \( \mu_0(\R^n) \) in the antepenultimate line. If \( \e'\leq r \leq R \), then it is enough to show that \( Y \) can be covered by \( C \) balls of radius \( \e' \), and such a finite \( C \) exists since \( Y \) is compact. 
	\end{proof}
	
	\begin{cor}
		\label{cor:approx}
		If \( E \) is an approximate tangent \( m \)-plane for \( X_0\setminus A \) at \( p\in X_0\setminus A \), and \( 0<\e<1 \), then there exists \( r>0 \) such that \( X_0(p,r)\setminus \overline{\cal{C}(p,E,\e)}=\emptyset \). 
	\end{cor}

	\begin{proof}
		By \eqref{eq:lowerbound} and Corollary \ref{cor:matapplied}, we may apply Lemma \ref{lem:approx} to the set \( X=X_0\setminus A \).
	\end{proof}

	\begin{thm}[Rectifiability]
		\label{thm:rectifiable}
		The measure \( \mu_0\lfloor_{A^c} \) is \( m \)-rectifiable.
	\end{thm}

	\begin{proof}
		By Corollary \ref{cor:matapplied}, it is enough to show that \( X_0\setminus A \) is \( m \)-rectifiable. Write \( X_0\setminus A=R\cup P \) where \( R \) is a \( m \)-rectifiable Borel set, \( P \) is purely \( m \)-unrectifiable and \( R\cap P=\emptyset \) (see e.g. \cite{mattila} 15.6.) If \( \H^m(P)>0 \), there exists by \cite{mattila} 6.2 a point \( p\in P \) such that \( 2^{-m}\leq {\Theta^*}^m(P,p)\leq 1 \) and \( {\Theta^*}^m(R,p)=0 \). Thus, for every \( \e>0 \) there exists \( 0<r<d_p/(2\sqrt{n}) \) such that \( \H^m(R(p,r\sqrt{n}))<\e \,r^m \) and \( \H^m(P(p,r)) > r^m \a_m 2^{-m-1} \). Let \( Q \) be a cube centered at \( p \) with side length \( 2r \). Then \( \H^m(P\cap Q)> \diam(Q)^m 2^{-2m-1}\a_m \) and 
		\begin{equation}
			\label{eq:rect0}
			\H^m(R\cap Q)<\e\cdot \diam(Q)^m.
		\end{equation}
		Since \( \H^m(X_0\setminus A)<\i \), there exists a cube \( Q' \) disjoint from \( \cal{N}(A,d_p/2) \) satisfying \( Q\subset \ring{Q'}, \)
		\begin{equation}
			\label{eq:rect1}
			\H^m(X_0\cap Q'\setminus Q)<\e\cdot \diam(Q')^m,
		\end{equation}
		\begin{equation}
			\label{eq:rect2}
			\H^m(P\cap Q)> \diam(Q')^m 2^{-2m-1} \a_m,
		\end{equation}
		and
		\begin{equation}
			\label{eq:rect6}
			\mu_0(\fr Q')=0.
		\end{equation}
		
		Now apply Lemma \ref{lem:ffapprox} to the cube \( Q' \) and the sets \( \text{``\( U \)''}=P\cap Q \) and \( \text{``\( E \)''}=X_0\cap Q' \). By Corollary \ref{cor:semiregular}, the set \( E \) is semiregular, so the Lipschitz constant of the resulting map \( \phi_1 \) is bounded above by some constant \( J<\i \), independent of the choice of \( \e \), \( Q \) or \( Q' \). Thus, by \eqref{eq:rect1}, \eqref{eq:rect0}, \eqref{lem:ffapprox:1}, and \eqref{eq:rect2},
 		\begin{align}
			\H^m(\phi_1(X_0\cap Q'))&\leq \H^m(\phi_1(X_0\cap Q'\setminus Q))+ \H^m(\phi_1(R\cap Q))+\H^m(\phi_1(P\cap Q))\\
			&\leq 2 J^m \e\cdot \diam(Q')^m\\
			&< \frac{\a_m}{2^{2m+2}} J^m \e \,\H^m(P\cap Q)\\
			&\leq \frac{\a_m}{2^{2m+2}} J^m \e \,\H^m(X_0\cap Q').\label{eq:rect4}
		\end{align}
		
		We will need to apply the map \( \phi_1 \) to \( X_k \) for \( k \) large. Since \( X_k\to X_0 \) in the Hausdorff metric, it follows from Lemma \ref{lem:FF} \ref{lem:ff:item:6} that there exists \( N<\i \) such that if \( k>N \), then 
		\begin{equation}
			\label{eq:rect3}
			\phi_1(X_k\cap Q')\subset S_{j,d}(Q')\cup \p(Q').
		\end{equation}

		For each \( k\geq 1 \) let \( \nu_k \) denote the measure \( \H^m\lfloor_{\phi_1(X_k\cap Q')} \) and let \( \nu_0 \) be a subsequential limit of \( \{\nu_k\}_{k\in \N} \).
		
		Note that \( \supp(\nu_0)\subset Q' \). More specifically, \( \supp(\nu_0)\subset \phi_1(X_0\cap Q') \), for if \( x\notin \phi_1(X_0\cap Q') \) and \( x\in Q' \), then since \( \phi_1 \) is proper, there is an open neighborhood \( V \) of \( x \) whose closure is disjoint from \( \phi_1(X_0) \). Thus, for large enough \( k \), we have \( \phi_1(X_k)\cap \overline{V}=\emptyset \). So, \( \nu_0(V)\leq \limsup \nu_k(\overline{V})=0. \) 
		
		Now suppose \( x\in \phi_1(X_0\cap Q') \). By \eqref{eq:rect3} it holds that for \( s>0 \) small enough,
		\begin{align*}
			\nu_0(B(x,s))&\leq \limsup \nu_k(\cal{N}(x,2s))\\
			&\leq \H^m(S_{j,d}(Q')\cup \fr(Q') \cap \cal{N}(x,2s))\\
			&\leq {n \choose m} \a_m (2s)^m.
		\end{align*}
		We conclude that \( {\Theta^*}^m(\nu_0,x)\leq {n\choose m}2^m \) for all \( x\in \supp(\nu_0) \) and so by \cite{mattila} 6.9 and \eqref{eq:rect4}, \[ \nu_0(\R^n)\leq {n\choose m}2^{2m} \H^m(\phi_1(X_0\cap Q')) \leq {n\choose m} 2^{4m+2}\a_m^{-1} J^m \e \,\H^m(X_0\cap Q'). \]
		
		Thus, there exists a sequence \( k_i\to \i \) with \( \H^m(\phi_1(X_{k_i}\cap Q'))\leq {n\choose m} 2^{4m+3}\a_m^{-1} J^m \e \,\H^m(X_0\cap Q'), \) and so
		\begin{equation}
			\label{eq:mainpoint}
			\F^m(\pi(\phi_1(X_{k_i}\cap Q'))) \leq T \e \,\H^m(X_0\cap Q'),
		\end{equation}
		where \( T\equiv b \lip(\pi)^m {n\choose m} 2^{4m+3}\a_m^{-1} J^m. \) Therefore,
		
		\begin{align}
			\label{eq:rect7}
			\F^m(\pi(\phi_1(X_{k_i}))\setminus A) \leq T \e \,\H^m(X_0\cap Q')+ \F^m(X_{k_i}\cap Q'^c\setminus A).
		\end{align}
		
		For large enough \( i \), by Corollary \ref{cor:densitybound}, weak convergence, \eqref{eq:rect6} and \cite{mattila} 6.9, \[ \diam(Q')^m \a_m/ 2^{2m+1} \leq \H^m(X_0\cap Q') \leq W \F^m(X_{k_i}\cap Q'), \] where \( W=2\a_m/(a\mathbf{c}). \) Thus for \( 0< \e < 1/(WT) \), \[ T \e \,\H^m(X_0\cap Q') \leq -(1/W - T \e) \diam(Q')^m \a_m/ 2^{2m+1} + \F^m(X_{k_i}\cap Q'). \]
		
		Together with \eqref{eq:rect7}, this implies \[ \F^m(\pi(\phi_1(X_{k_i}))\setminus A) \leq \F^m(X_{k_i}\setminus A) -(1/W - T \e) \diam(Q')^m \a_m/ 2^{2m+1} \] a contradiction for \( i \) large enough, since \( \pi(\phi_1(X_{k_i}))^\dagger\in \Span(C,A) \) and \( \F^m(X_{k_i}\setminus A)\to \frak{m} \).
	\end{proof}	
	
\section{Lower semicontinuity}

	Given \( 0< \e < 1 \) and \( p\in \R^n \), let \( A(p,\e,r) \) denote the closed annular region \( B(p,r) \setminus \cN(p,(1-\e)r) \). For \( p\in A^c \), let \( C_{m,p} = 2c(p)m\frac{b}{a} \).

	\begin{lem}
		\label{lem:newhalfradius}
		Let \( 0 < \e < 1/2 \) and \( p\in X_0\setminus A \). There exist sequences \( \{r_i\}_{i\in \N}\subset D_p \) and \( \{ N_i \}_{i\in \N}\subset \N \) with \( r_i \to 0 \) such that \[ \H^m(X_k \cap A(p,\e,r_i)) \le C_{m,p}\e r_i^m \] for all \( k \ge N_i \).
	\end{lem}

	\begin{proof}
		If not, there exists \( \d > 0 \) such that for every \( r \in D_p\cap (0,\d) \) there exists a subsequence \( k_j\to \i \) with
		\begin{equation}
			\label{kstuff}
			\F^m(X_{k_j}\cap A(p,\e,r)) > a C_{m,p}\e r^m,
		\end{equation}
		and hence
		\begin{equation}
			\label{stuff}
			\mu_0(A(p,\e,r)) \geq a C_{m,p}\e r^m.
		\end{equation}
		Indeed, since \( D_p \) is dense in \( (0,d_p) \), \eqref{stuff} holds for all \( r\in (0,\min\{d_p,\d\}) \).

		Since \( \mu_0 \) is finite, there exist arbitrarily small \( r\in (0,\min\{d_p,\d\}) \) such that \( \mu_0(\fr A(p,\e,r_i))=0 \) for all \( i\geq 0 \), where \( r_i \equiv (1-\e)^i r \). For such \( r \), it follows from \eqref{stuff} that
		\begin{align*}
			\mu_0(B(p,r)) &= \sum_{i=0}^\i \mu_0(A(p,\e,r_i))\\
			&\geq a C_{m,p} \e \sum_{i=0}^\i (1-\e)^{im} r^m\\
			&= a C_{m,p}\frac{\e}{1-(1-\e)^m} r^m\\
			&\ge 2c(p)b r^m,
		\end{align*}
		contradicting Theorem \ref{thm:upperdensity}.
	\end{proof}

	\begin{thm}
		\label{thm:lsc}
		If \( V\subset A^c \) is open, then \[ \F^m(X_0 \cap V) \le \liminf_{k\to\i} \F^m(X_k \cap V). \]
	\end{thm}

	The idea of the proof is to form a Vitali-type covering of \( X_0 \cap V \) by disjoint balls \( B(p_i,r_i) \) where \( X_0 \) has an approximate tangent plane \( E_i \) at \( p_i \). The ellipticity condition on \( f \) will provide a lower bound for \( \F^m(X_k(p_i,r_i)) \) in terms of \( f(p_i,E_i) \a^m r_i^m \) and a small error. Lusin's theorem and rectifiability of \( \mu_0 \) will provide an upper bound for \( F^m(X_0(p_i,r_i)) \), again in terms of \( f(p_i,E_i) \a^m r_i^m \) and a small error.

	\begin{proof}
		Since \( \F^m(X_k\setminus A)<\i \), \( k\geq 0 \), it suffices to prove the claim for \( V \) disjoint from a neighborhood of \( A \). Let \( X_0' \) be the full \( \H^m \) measure subset of \( X_0\cap V \) consisting of those points \( p \) such that \( \Theta^m(X_0, p)=1 \) and for which \( X_0 \) has an approximate tangent \( m \)-plane \( E_p \) at \( p \).
		
		Fix \( 0<\e<1/2 \). By Lusin's theorem there exists an \( \H^m \) measurable \( Z \subset X_0' \) for which the function \( q \mapsto f(q,T_q X_0) \) is continuous and 
		\begin{equation}
			\label{eq:lsc111}
			\F^m(X_0' \setminus Z)\le \e.
		\end{equation}
		Let \( p \in Z \) and write \( E\equiv p + E_p \). By Lemmas \ref{lem:approx} and \eqref{eq:lowerbound} there exists \( 0 < d_p' < \min\{d_p, d_H(p,V^c)\} \) such that if \( 0 < r < d_p' \), then
		\begin{equation}
			\label{Neighborhood}
			X_0(p,r)\subset \cal{N}(E, \e r/2)
		\end{equation}
		and
		\begin{equation}
			\label{eq:lscdensitybd}
			(1-\e) \a_m r^m\leq \H^m(X_0(p,r))\leq (1+\e) \a_m r^m.
		\end{equation}

		By Lemma \ref{lem:newhalfradius} there exist sequences \( \{N_j\}_{j\in \N} \subset \N \) and \( \{r_j\}_{j\in \N}\subset D_p\cap (0,d_p') \) with \( r_j \to 0 \) such that
		\begin{equation}
			\label{eq:annulusestimate}
			\H^m(X_k \cap A(p,\e,r_j)) \le C_V\e r_j^m,
		\end{equation}
		for all \( k \ge N_j \), where \( C_V\equiv \sup_{p\in X_0\cap V} \{ C_{m,p}\}<\i \).

		\begin{description}
			\item[An upper bound for \( \F^m(X_0(p,r_j)) \)]

				By \eqref{eq:lscdensitybd}, for large enough \( j \), we have

				\begin{align*}
					\F^m(Z(p,r_j)) &= \int_{Z(p,r_j)} f(q,T_q X_0) d\H^m\\
					&\le (f(p,E_p) + \e)(\H^m(X_0(p,r_j))\\
					&\le (f(p,E_p) + \e)(\a_m r_j^m + \e r_j^m)\\
					&\le f(p,E_p)\a_m r_j^m + K \e r_j^m,
				\end{align*}
				where \( K<\i \) is independent of \( p \), \( j \) and \( \e \). Thus,
				\begin{equation}
					\label{eq:x0E}
					\F^m(X_0(p,r_j)) \leq f(p,E_p)\a_m r_j^m + K \e r_j^m + \F^m((X_0'\setminus Z)(p,r_j)).
				\end{equation}
			
			\item[A lower bound for \( \F^m(X_k(p,r_j)) \)]

				Fix \( j\in \N \). Since \( X_k\to X_0 \) in the Hausdorff metric, we may increase \( N_j \) if necessary so that
				\begin{equation}
					\label{eq:annulusestimate2}
					X_k(p,r_j) \subset \cal{N}(E,\e r_j/2)
				\end{equation}
				for all \( k\ge N_j \).

				For each \( j\in \N \) there exists a Lipschitz map \( \phi_j:\R^n\to \R^n \) such that
				\begin{enumerate}
					\item \( \phi_j \) is the identity outside \( A_j\equiv A(p,\e,r_j)\cap \cal{N}(E, \e r_j) \);
					\item \( \phi_j(A_j)=A_j \);
					\item On \( A(p,\e/3,(1-\e/3)r_j)\cap \cal{N}(E, \e r_j/2)) \), the map \( \phi_j \) is orthogonal projection onto \( E \);
					\item \( \phi_j \) is \( C^0 \)-close to a diffeomorphism;
					\item The Lipschitz constant of \( \phi_j \) depends only on \( n \).
				\end{enumerate}

				By Corollary \ref{cor:densitybound} and \eqref{eq:annulusestimate} it holds that for small enough \( \e>0 \), we may increase the constant \( N_j \) so that if \( k \ge N_j \), then \( \phi_j(X_k) \) contains \( E\cap \p B(p, (1-\e/2)r_j) \). Indeed, if \( E\cap \p B(p, (1-\e/2)r_j) \) contains a point \( q \) which is not in \( \phi_j(X_k) \), then we may orthogonally project \( \phi_j(X_k(p,(1-\e/2)r_j)) \) onto \( E \), and then radially project the resulting set away from \( q \) onto \( E\cap \p B(p, (1-\e/2)r_j) \). For \( \e>0 \) small enough, the image of \( \phi_j(X_k) \) by this map will be contained in the neighborhood \( U \) of \( C \), and we may apply the Lipschitz retraction \( \pi:U \to C \) to create a new sequence \( \{Y_k\} \subset \Span(C,A) \). The sets \( Y_k \) satisfy \( Y_k=X_k\setminus B(p,r_j)\cup Z_k \), where \( \H^m(Z_k)\leq \kappa^m C_V \e r_j^m \), and \( \kappa<\i \) depends only on the Lipschitz constants of \( \phi_j \) and \( \pi \). Using the density bounds in Corollary \ref{cor:densitybound}, we conclude that \( \inf_k \{\F^m(Y_k)\}<\frak{m} \) yielding a contradiction (c.f. the proof of Theorem \ref{thm:rectifiable} after \eqref{eq:rect7}.)

				Moreover, there can be no retraction from \( \phi_j(X_k)(p, (1-\e/2)r_j) \) onto \( E\cap \p B(p, (1-\e/2)r_j) \), for if there exists such a retraction \( \rho \), then by the Stone-Weierstass theorem for locally compact spaces, it can be assumed without loss of generality to be Lipschitz. As a Lipschitz map, \( \rho \) can then be extended by the identity to \( \p B(p, (1-\e/2)r_j) \), then to the rest of \( B(p, (1-\e/2)r_j) \) using the Kirszbraun theorem, and then finally by the identity to all of \( \R^n \). The sequence \( \{\pi(\rho(X_k))\}\subset \Span(C,A) \) will yield a contradiction for the same reason as above.

				Thus, by the ellipticity of \( f \) and \eqref{eq:annulusestimate}, it holds that for large enough \( j \) and \( k\ge N_j \),
				\begin{align*}
					(f(p,E_p)-\e )\a_m (1-\e/2)^m r_j^m &\leq \F_{\tilde{f}}^m(\phi_j(X_k)(p, (1-\e/2)r_j))\\
					&\leq \F^m(X_k(p, r_j)) + b' \lip(\phi_j)C_V\e r_j^m,
				\end{align*}

				or in other words,
				\begin{equation}
					\label{eq:lsclowerbound}
					f(p,E_p)\a_m r_j^m \leq \F^m(X_k(p, r_j)) + K \e r_j^m,
				\end{equation}
				where \( K<\i \) is independent of \( p \), \( j \) and \( \e \).

			\item[A Covering]
				
				By \cite{mattila} Theorem 2.8, \eqref{eq:x0E} and \eqref{eq:lsclowerbound}, there exists a covering \( \{B(p_i,s_i)\}_{i\in I} \) of \( \H^m \) almost all \( Z \) by disjoint closed balls \( B(p_i,s_i) \) with \( p_i \in Z \) and \( 0<s_i< d_p' \) small enough so that
				\begin{equation}
					\label{eq:x0E2}
					\F^m(X_0(p_i,s_i)) \le f(p_i,E_{p_i})\a_m s_i^m + K \e s_i^m + \F^m((X_0'\setminus Z)(p_i,s_i))
				\end{equation}
				and
				\begin{equation}
					\label{eq:lsclowerbound2}
					f(p_i,E_{p_i})\a_m s_i^m \leq \F^m(X_k(p_i, s_i)) + K \e s_i^m,
				\end{equation}
				for \( k \) large enough (depending on \( i \).)
				
				Choose a finite subcover \( \{B(p_i,s_i)\}_{i=1}^N \) such that
				\begin{align}
					\label{coveringbyballs}
					\F^m(Z \setminus \cup_{i=1}^N Z(p_i,s_i)) < \e.
				\end{align}
				
				Associated to each \( i=1,\dots,n \) is the constant \( N_i \) from \eqref{eq:annulusestimate}. Let \( N = \max_i \{N_i\} \). By \eqref{coveringbyballs}, \eqref{eq:x0E2}, \eqref{eq:lscdensitybd}, \eqref{eq:lsclowerbound2}, and \eqref{eq:lsc111},
				\begin{align*}
					\F^m(X_0 \cap V) &\leq 2\e + \sum_{i=1}^N \F^m(X_0(p_i,s_i))\\
					&\leq 2\e + \sum_{i=1}^N f(p_i, E_{p_i})\a_m s_i^m + K \e s_i^m + \F^m((X_0'\setminus Z)(p_i,s_i))\\
					&\leq 3\e + \sum_{i=1}^N \F^m(X_k(p_i, s_i)) + 2K \e s_i^m\\
					&\leq 3\e + 4K \e \H^m(X_0\cap V)/\a_m + \F^m(X_k \cap V),
				\end{align*}
				for all \( k \ge N \). Therefore, \[ \F^m(X_0 \cap V) \le \liminf_{k \to \i} \F^m(X_k\cap V) + C \e, \] where \( C<\i \) is independent of \( \e \). Since this holds for all \( \e > 0 \), the result follows.
		\end{description}
	\end{proof}

	In particular, \( \F^m(X_0\setminus A) = \frak{m} \). This completes the proof Theorem \ref{thm:main}.

		 	\addcontentsline{toc}{section}{References} 
		 	\bibliography{bibliography.bib, mybib.bib}{}
		 	\bibliographystyle{amsalpha}
			\end{document}